\documentclass[11pt,twoside]{amsart}

\usepackage[dvipsnames]{xcolor}
\usepackage{hyperref}
\hypersetup{colorlinks=true, linkcolor=Blue, citecolor=Maroon}
\usepackage{amsthm, amsmath, amscd, amssymb,centernot,txfonts}
\usepackage{tikz-cd}

\usepackage{lipsum}
\usepackage{multicol}
\usepackage{caption}

\usepackage[normalem]{ulem}
\usepackage{cancel}
\usepackage{amsfonts}
\usepackage{mathrsfs}
\usepackage{parskip}
\usepackage[all]{xy}
\usepackage[left=2.5cm,top=2.5cm,bottom=3cm,right=2.5cm]{geometry}
\usepackage{dirtytalk}
\usepackage{verbatim}
\usepackage{mathtools}
\usetikzlibrary{shapes,arrows}
\usepackage{etoolbox}
\usepackage{dynkin-diagrams}

\usepackage{enumitem}
\usepackage{chngpage}

\usepackage{adjustbox}
\usepackage[utf8]{inputenc}
\usepackage{fourier} 
\usepackage{array}
\usepackage{makecell}

\usepackage[linktocpage=true]{}

\usepackage{color, colortbl}
\definecolor{LightCyan}{rgb}{0.88,1,1}

\usepackage[normalem]{ulem}

\usepackage[first=0,last=9]{lcg}

\newcommand{\floor}[1]{\left\lfloor #1 \right\rfloor}

\newcommand{\stkout}[1]{\ifmmode\text{\sout{\ensuremath{#1}}}\else\sout{#1}\fi}

\theoremstyle{plain}
\numberwithin{equation}{section}
\newtheorem{theorem}{Theorem}[section]
\newtheorem{proposition}[theorem]{Proposition}
\newtheorem{lemma}[theorem]{Lemma}
\newtheorem{corollary}[theorem]{Corollary}

\newtheorem{set-up}[theorem]{Set-up}

\theoremstyle{definition}
\newtheorem{remark}[theorem]{Remark}

\newtheorem{example}[theorem]{Example}
\newtheorem{definition}[theorem]{Definition}
\newtheorem{question}[theorem]{Question}

\newcommand{\PP}{\mathbb P}

\newcommand*{\QEDB}{\hfill\ensuremath{\square}}

\tikzstyle{decision} = [diamond, draw, , 
    text width=4.5em, text badly centered, node distance=3cm, inner sep=0pt]
\tikzstyle{block} = [rectangle, draw, , 
    text width=10em, text centered, rounded corners, minimum height=2em]
\tikzstyle{block1} = [rectangle, draw, , 
    text width=5em, text centered, rounded corners, minimum height=2em]    
\tikzstyle{line} = [draw, -latex']
\tikzstyle{cloud} = [draw, ellipse,, node distance=3cm,
    minimum height=2em]

\begin{document}

\title[Construction of varieties of low codimension with application to moduli spaces]{Construction of varieties of low codimension with applications to moduli spaces of varieties of general type}

\author[P. Bangere]{Purnaprajna Bangere}
\address{Department of Mathematics, University of Kansas, Lawrence, USA}
\email{purna@ku.edu}
\author[F.J. Gallego]{Francisco Javier Gallego}
\address{Departamento de \'Algebra, Geometr\'ia y Topolog\'ia and Instituto de Matem\'atica Interdisciplinar,
Universidad Complutense de Madrid, Madrid, Spain}
\email{gallego@mat.ucm.es}
\author[J. Mukherjee]{Jayan Mukherjee}
\address{Department of Mathematics, University of California Riverside, Riverside, USA}
\email{jayanm@ucr.edu}
\author[D. Raychaudhury]{Debaditya Raychaudhury}
\address{Department of Mathematics, University of Toronto, Toronto, Canada}
\email{debaditya.raychaudhury@utoronto.ca}
\subjclass[2020]{14B10, 14D06, 14D15, 14D20, 14M10}
\keywords{Deformations of polarized varieties, deformations of morphisms, multiple structures, complete intersections, Fano varieties, Calabi-Yau varieties, varieties of general type, moduli of varieties of general type.}

\maketitle

\begin{center}
    \textit{Dedicated to our dear friend Miguel Gonz\'alez on his 60th birthday}
\end{center}

\begin{abstract}
\textcolor{black}{In this article we develop} a new way of systematically constructing \textcolor{black}{infinitely many} families of smooth subvarieties $X$ of \textcolor{black}{any given} dimension $m$, $m \geq 3$, and any \textcolor{black}{given} codimension in \textcolor{black}{{$\mathbb P^N$}}, embedded by complete subcanonical linear series, and,
in particular, in the range of Hartshorne's conjecture. We accomplish this by showing the existence of everywhere non--reduced schemes called ropes, embedded in {{$\mathbb P^N$}}, and by smoothing them. In the range $3 \leq m < {{N/2}}$, we construct smooth subvarieties, embedded by complete subcanonical linear series, that are not complete intersections. 
We also go beyond a question of Enriques on constructing simple canonical surfaces in projective spaces, and construct simple canonical varieties in all dimensions. \textcolor{black}{The canonical map of infinitely many} of these simple canonical
varieties \textcolor{black}{is}  finite birational \textcolor{black}{but} not an embedding. Finally, we show the existence of components of moduli spaces of varieties of general type (in all dimensions $m$, $m \geq 3$) that are analogues of the moduli space of curves of genus $g > 2$ with respect to the behavior of the canonical map \textcolor{black}{and its deformations}. In many cases, the general elements of these components are canonically embedded and their codimension \textcolor{black}{is} in the range of Hartshorne's conjecture.
\end{abstract}

\normalsize
\section{Introduction}

In this article, we study the deformations of morphisms $\varphi$  to projective spaces that factor through a finite cover $\pi$ of degree $n$ of a complete intersection \textcolor{black}{subvariety}. As a result of this, we find a new, systematic method to construct infinitely many smooth subvarieties of \emph{any} codimension, embedded \textcolor{black}{in any projective space $\mathbb{P}^N$ by complete linear series,} including infinitely many smooth, non--complete intersection subvarieties and infinitely many simple canonical varieties \textcolor{black}{in \textcolor{black}{any dimension $m$, $m \geq 3$.}} We also find, for varieties of general type of \textcolor{black}{any dimension} $m \geq 3$, irreducible 
moduli components  with a closed locus where the degree of the canonical map jumps up. This includes, for any $m \geq 3$, components whose  general points correspond to varieties whose canonical map is either birational or an embedding.
As we explain in more detail below, the \textcolor{black}{sub}varieties we construct and, in particular, the smooth canonical varieties corresponding to general points of the above mentioned moduli components, degenerate to certain everywhere non--reduced schemes of multiplicity $n$, called \emph{ropes}, as it happens in the moduli of curves.

\smallskip

To carry out this program, we find out the conditions under which the deformation of $\varphi$ can be deformed to an embedding or, more generally, to a morphism of lower degree, along a one-parameter family.
In order for this to happen, 
 we show
 first the existence  of ropes  of the right multiplicity and codimension embedded in
  \textcolor{black}{$\mathbb P^N$}, and prove that they can be smoothed. In particular, 
 the smooth subvarieties constructed by  our method are smoothings of these ropes. \textcolor{black}{\textcolor{black}{Therefore} we produce one--parameter families whose general members are smooth subvarieties and whose special member is not even a local complete intersection, except \textcolor{black}{if it is a rope} of multiplicity $2$.} As already said, 
our study of the deformations of $\varphi$ goes beyond the case of deformations to embeddings. Indeed, we prove general results (see Proposition \ref{nonexistence} and Theorems \ref{birational}, \ref{theorem.embedding}, \ref{2:1})
{which show} that the deformations of $\varphi$ have very diverse behavior and the degree of the morphisms so obtained varies greatly (from degree $1$ to degree $n/2$ and $n$). In particular, we obtain criteria for $\varphi$ to be deformed to birational morphisms which are not necessarily embeddings. Although we extensively use our general results  when $\pi$ is a simple cyclic cover, a cover acted by $\mathbb Z_{n/2} \times \mathbb Z_2$ or a dihedral cover \textcolor{black}{of a complete intersection subvariety}, this article shows the way to follow
when $\pi$ is \textcolor{black}{an arbitrary} Galois cover \textcolor{black}{or, even, certain finite covers  of complete intersections, which are not necessarily Galois (see Set-up~\ref{setup1})}. We detail now the several applications, given in the paper, of our method and techniques.  

\smallskip

\noindent{\bf Construction of smooth, small codimension subvarieties.} Among the smooth subvarieties we construct,  there are  infinitely many families of  subvarieties of small codimension $r$, embedded by complete linear series in $\mathbb{P}^N$. They include subvarieties in the range
$r <\frac{1}{3}N$, that is, in the range of Hartshorne's conjecture on complete intersections. Even if  some of these subvarieties turn out to be complete intersections (see Corollary~\ref{kpr} and 
Proposition~\ref{prop.ci}), there are some others for which this is not clear. Among them, those at or near the boundary of the range of Hartshorne's conjecture could be especially interesting. In the following two tables 
we give a very small sample of the smooth subvarieties  we construct in the range $r=(1/3)N-1$. 
In the tables, $X'$ is a smooth projective $s$--subcanonical (see Definition \ref{defsubcan}), $m$--dimensional subvariety, obtained by deforming $\varphi$,  where $\varphi$ factors through a finite cover of degree $n$ of a complete intersection $Y$ of multidegree $\underline{d}$. 
More precisely, in Table \ref{t01}, the subvarieties are obtained by deforming simple cyclic covers branched along a smooth divisor in $|\mathcal O_Y(2n)|$. In Table \hyperref[t02]{2}, the subvarieties are obtained by deforming simple  
\textcolor{black}{dihedral} covers with $k=2$ (see Section \ref{5} for notation).

 \small
 
 \begin{multicols}{2}
  \begin{center}\phantomsection\label{t01}
 \begin{tabular}{c|c|c|c|c|c}
 \hline
 $m$ & $n$ &  $N$ & $s$ & $\underline{d}$ & {deg}($X'$)\\
 \hline\hline
 $9$ & $4$ & $12$ & $5$ & $(2,4,6)$ & $192$ \\
 \hline
 $11$ & $5$ &  $15$ & $12$ & $(2,4,6,8)$ & $1920$\\
 \hline
 $15$ & $7$ &  $21$ & $32$ & $(2,4,6,8,10,12)$ & $322560$\\
 \hline
\end{tabular}
\captionof{table}{{Deforming \textcolor{black}{simple cyclic} covers to obtain low codimension subvarieties}}
\end{center}\columnbreak
\begin{center}\phantomsection\label{t02}
 \begin{tabular}{c|c|c|c|c|c}
 \hline
 $m$ & $n$  & $N$ & $s$ & $\underline{d}$ & {deg}($X'$)\\ 
 \hline\hline
 $13$ & $6$ &
 $18$ & $5$ & $(2,2,4,4,6)$ & $2304$ \\
 \hline
 $17$ & $8$ &  $24$ & $15$ & $(2,2,4,4,6,6,8)$ & $147456$ \\
 \hline
 $21$ & $10$ & $30$ & $29$ & $(2,2,4,4,6,6,8,8,10)$ & $14745600$ \\
 \hline
 \end{tabular}
\captionof{table}{{Deforming \textcolor{black}{simple dihedral} covers to obtain low codimension subvarieties}}
\end{center}
\end{multicols}
\normalsize
\noindent 
The general deformation of the morphisms $\varphi$  of Table \ref{t01}  is an embedding. By Proposition~\ref{prop.ci}, the image of this general deformation of $\varphi$ is a complete intersection. However, we do not know whether the same is true for  some special deformations (see Question~\ref{question.ci}). In the case of Table \hyperref[t02]{2}, we do not know whether the images of the deformations of $\varphi$ are complete intersections subvarieties or not (see Question~\ref{question.ci.diherdral}). For comprehensive results on the construction of small codimension, smooth subvarieties, see Sections~\ref{4} and \ref{5} and Subsection~\ref{embznz2}.

\smallskip

\noindent{\bf  Construction of smooth, non--complete intersection subvarieties.}
We also construct, in a systematic way, smooth subvarieties, embedded in $\mathbb P^N$ by complete linear series, which are not complete intersections. Among the  subvarieties of the lowest dimension and degree that we obtain are the threefolds 
whose invariants are detailed in the following table, where $X'$ is $s$-subcanonical and is obtained by deforming $\varphi$, which factors through a simple cyclic \textcolor{black}{cover} of degree $n$ branched along a smooth divisor in $|\mathcal O_Y(2n)|$ (see Example~\ref{emb.more} for further details):

\vskip .4cm

\begin{center}\phantomsection\label{t03}
 \begin{tabular}{c|c|c|c|c} 
 \hline
  $n$ &  $N$ & $s$ & $\underline{d}$ & 
 deg$(\varphi'(X'))$\\
 \hline\hline
  $2$ & \color{black} $7$ & $6$ &  $(3,3,3,3)$ &  $162$\\
 \hline
 $2$ &  $7$ & $7$ &  $(3,3,3,4)$ & $216$\\
 \hline
   $3$ &  $8$ & $15$ & $(4,4,4,4,4)$ & $3072$\\
 \hline
   $3$ &  $8$ & $16$ & $(4,4,4,4,5)$ & $3840$\\
    \hline
 \end{tabular}
 \captionof{table}{Deforming \textcolor{black}{simple cyclic} covers to obtain non--complete intersections}
 \end{center}
 
 \vskip .4cm

\noindent More generally, for any  $3 \leq m < \frac{N}{2}$, 
we construct non--complete intersection, $m$--dimensional smooth subvarieties of $\mathbb P^N$ of infinitely many different degrees (see Theorems~\ref{real.main0.5} and \ref{main0.5}). \textcolor{black}{A standard way of constructing smooth non--complete intersection subvarieties in this range is to realize them as degeneracy loci of vector bundle homomorphisms. Our method is quite different from that, \textcolor{black}{since our} smooth subvarieties \textcolor{black}{come} as general members of families \textcolor{black}{which smooth ropes}.}

\smallskip

\noindent{\bf  Construction of simple canonical varieties.}
In 1943 Enriques  raised the question of the existence of simple canonical surfaces  in projective spaces, i.e., surfaces of general type for which the canonical map is birational. 
In the present article we systematically construct simple canonical varieties of general type in all dimensions $m \geq 3$. More generally, 
we deform the morphisms $\varphi$ to morphisms $\varphi'$, birational onto their image, from smooth, $s$--subcanonical varieties $X'$. When $s=1$, $X'$ are simple canonical varieties. In most of those cases, the morphisms $\varphi'$ are embeddings and the varieties
$\varphi'(X')$ are canonically embedded, smooth subvarieties; however, we produce infinitely many examples of smooth varieties equipped with \textcolor{black}{finite} birational canonical maps, in fact, morphisms, that are not embeddings (see  Theorem~\ref{prop.birational} and Example \ref{enr}).

\smallskip

\noindent{\bf Moduli of varieties of general type.} \textcolor{black}{Deep results on} existence \textcolor{black}{and boundedness} of moduli spaces of varieties of general type  have been accomplished in recent years (see \cite{Kol}, \cite{Kol2}, \cite{Kov}, \cite{HMX}). The moduli space of curves $\mathcal{M}_g$ and its compactification has been a chief inspiration \textcolor{black}{for them and} \textcolor{black}{considering} one--parameter \textcolor{black}{families} of varieties of general type has \textcolor{black}{also} been crucial. \textcolor{black}{For each $m \geq 3$} it is natural to \textcolor{black}{ask if one} can systematically construct nontrivial examples of moduli spaces of \textcolor{black}{$m$--dimensional} varieties of general type $X$ that have components analogous to \textcolor{black}{the} moduli \textcolor{black}{space} of curves,  with respect to the behavior of canonical maps and their deformations. To be precise, \textcolor{black}{one would like to construct moduli components with locally closed loci that resemble   the hyperelliptic locus of $\mathcal{M}_g$ ($g >2$) in the following sense.} \textcolor{black}{The hyperelliptic locus parametrizes curves whose} canonical maps  \textcolor{black}{are} finite \textcolor{black}{morphisms $\varphi$} of degree $2$ onto its image \textcolor{black}{such that} {a} general deformation \textcolor{black}{of $\varphi$} is \textcolor{black}{an embedding}. More subtly, {a} general deformation of \textcolor{black}{$\varphi$} gives rise to \textcolor{black}{a} one--parameter  \textcolor{black}{family of subvarieties} \textcolor{black}{\textcolor{black}{whose central member is a rope of multiplicity $2$ (a canonical ribbon)}, while
\textcolor{black}{its} general \textcolor{black}{member} is a smooth canonically embedded curve.} 

\smallskip

\textcolor{black}{In Section \ref{section.moduli}, we produce 
moduli components of higher dimensional varieties of general type that capture the features, described above, of the moduli space of curves and its hyperelliptic locus. Indeed,} 
we show how to systematically construct nontrivial examples of moduli \textcolor{black}{spaces} of varieties of general type (for \textcolor{black}{instance,} not \textcolor{black}{being products})  \textcolor{black}{of dimension $m$, $m \geq 3$,} with an \textcolor{black}{irreducible} component having a \textcolor{black}{locally} closed \textcolor{black}{locus} that
\textcolor{black}{parametrizes varieties whose canonical map is a} finite \textcolor{black}{morphism $\varphi$
} of degree $\textcolor{black}{n, n}\geq 2$, onto its image \textcolor{black}{such that} its general deformation \textcolor{black}{is an embedding}. Moreover, like in the case of curves, the image of \textcolor{black}{the central fiber of} \textcolor{black}{a general} first order deformation \textcolor{black}{$\varphi$} is an embedded rope of multiplicity $n$. \textcolor{black}{Such a rope is the one--parameter degeneration of smooth, canonically embedded varieties that correspond to general points of the moduli component.
We find these canonically embedded varieties for any codimension $r$ and, in particular, for any $r$ in the range of the Hartshorne’s conjecture.}
\textcolor{black}{For each $m \geq 3$,} we also construct \textcolor{black}{two distinct kinds of moduli} components which differ from the moduli space of curves. Firstly, we exhibit moduli components (see Corollary \ref{def of can morphisms birational}) such that {a general {deformation}  of a \textcolor{black}{canonical} morphism $\varphi$ corresponding to a point in the special \textcolor{black}{locus} {is}  not {an} embedding but {a} birational morphism}, and the image of \textcolor{black}{the central fiber of} \textcolor{black}{a general} first order deformation of $\varphi$ is not an embedded rope, \textcolor{black}{but \textcolor{black}{a possibly locally non--Cohen--Macaulay,  everywhere non--reduced scheme.}} Secondly, we also construct moduli components (see Corollary \ref{def of can morphisms halved}) where the degree of the canonical map $\varphi$ drops from $n$ to $n/2$, as $\varphi$ deforms from \textcolor{black}{a locally closed locus} to the general stratum. For surfaces, the existence of moduli components with \textcolor{black}{locally} closed loci where the degree of the canonical \textcolor{black}{map} jumps up  \textcolor{black}{was} previously known (see \cite{Cat81}, \cite{Cat87}, \cite{CS02},  \cite{CPT00}, \cite{AK90}, \cite{GGP2}, \cite{GGP1} and \cite{moduli}).
\textcolor{black}{The  special loci seen in Section~\ref{section.moduli} are but a particular case of a more general phenomenon. In fact, our results imply (see Remarks~\ref{remark.moduli}, \ref{remark.moduli2}, \ref{remark.moduli3}, \ref{remark.moduli4}) the existence of infinitely many irreducible components possessing loci which  \textcolor{black}{correspond} to  \textcolor{black}{the jumping} up of the degree of subcanonical maps.  {These loci for higher dimensional varieties motivate  interesting questions concerning the moduli space of curves itself, namely, if the moduli of curves has analogous locally closed loci, in that case, in relation to the possible change of degree of morphisms induced by theta-characteristics or by other subcanonical divisors.}}
 
 \smallskip

\noindent{\bf Hilbert scheme components with ropes in their boundary.} 
As already mentioned, the construction and smoothing of ropes  is key in our method to produce subvarieties in projective space (see Theorem~\ref{theorem.embedding}). 
A priori, it is not clear why embedded ropes of codimensions in the range of  Hartshorne's conjecture and beyond should deform to smooth subvarieties. When these ropes  have multiplicity $n \geq 3$,  they are \textcolor{black}{not even local complete intersections.} We systematically construct, in any codimension (this includes the range of Hartshorne's conjecture), non--complete intersection, embedded ropes, lying in the boundary of an irreducible component 
of the Hilbert scheme parametrizing  smooth complete intersection 
subvarieties (see Proposition \ref{prop.ci}). We do this  when $\varphi$ satisfies Theorem~\ref{theorem.embedding.a} (a) and factors through a simple cyclic cover or, more generally, a suitable 
composition  of simple cyclic covers.
However, it  is not clear that  any one-parameter deformation of these ropes is a complete intersection, except if the codimension is $r \leq 2$, {where we use 
\textcolor{black}{the results in} \cite{KPR} to prove Corollary ~\ref{kpr} (\textcolor{black}{indeed, the question of complete intersections is better known for codimension $2$ subvarieties of projective space, see e.g. \cite{Hor}, \cite{KPR}, and  has even been studied in more general varieties, see the works \cite{Ott}, \cite{AC}, \cite{Mad}, \cite{CM}, \cite{KR071}, \cite{KR072} \cite{R09}, \cite{RT19}, to name just a few).}} Furthermore, when $\varphi$ factors through other Galois covers, such as, for instance, dihedral covers, it is also unclear that the general deformation of such ropes is a complete intersection subvariety. The situation is further complicated by the existence of examples (see 
Remark ~\ref{f3}) where we show that $\varphi$ is unobstructed 
in its deformation space, but the rope that appears when we deform $\varphi$
is obstructed, in some of the cases because it lies in 
the intersection of two components of the Hilbert scheme.
 \textcolor{black}{In addition,} in the range $r < N/2$, we also show the existence of  ropes that can be deformed to smooth subvarieties but do not lie in the boundary of any irreducible component of the Hilbert 
scheme  parameterizing smooth complete intersection subvarieties (see Theorems \ref{real.main0.5}, \ref{main0.5}). 

\smallskip

\noindent\textbf{Organization.} We now provide the structure of this article. In Section \ref{secprelim}, we recall the definition of ropes and several results, including the results on the deformation theory of finite morphisms developed by the first two authors and Miguel Gonz\'alez. In Section \ref{3} we study the deformations of finite covers of complete intersection subvarieties. In Section \ref{secabel} we give the details of the construction of  abelian and dihedral covers. Section \ref{secnonci} is devoted to finding the necessary and sufficient conditions to ensure that a general deformation of a finite cover is not a complete intersection. We apply general results of Sections \ref{3}, \ref{secabel} and \ref{secnonci} to deform  simple cyclic covers to  embeddings of smooth subvarieties of any dimension (including small codimensional subvarieties),  embeddings of non--complete intersections, and birational maps and, more generally, to non--embeddings, in Sections \ref{4}, \ref{section.cyclic.covers.nci}, and \ref{section.cyclic.covers.non.embeddings} respectively. We study the deformations of $\mathbb{Z}_{n/2}\times \mathbb{Z}_2$ and simple dihedral covers of complete intersections in Sections \ref{secz} and \ref{5}. In Section~\ref{section.moduli}, we apply the results of the previous sections to prove the existence of irreducible \textcolor{black}{moduli} components  \textcolor{black}{that behave like the moduli space of curves}. Finally, we devote Section~\ref{6} to open questions.

\smallskip

\noindent\textbf{Acknowledgements.} We sincerely thank Enrique Arrondo, \textcolor{black}{S\'andor Kov\'acs}, Mohan Kumar, and Madhav Nori for some motivational discussions and for generously sharing their time listening to our results. 
In particular, we thank Madhav Nori for insightful \textcolor{black}{conversations} 
on deformations of iterated simple cyclic covers from a point of 
view different from the one taken here; Mohan Kumar, on 
codimension two subvarieties and Hartshorne's conjecture; Enrique Arrondo, on low codimension subvarieties and different aspects of Hartshorne's conjecture; and S\'andor Kov\'acs, on issues related to moduli spaces. 
We also thank Miguel Gonz\'alez for helpful and motivating conversations. The second author was partially supported by Spanish Government grant  PID2021-124440NB-I00 and by Santander-UCM grant PR44/21. The third author was supported by the National Science Foundation, Grant No. DMS-1929284 while in residence at ICERM in Providence, RI, as part of the ICERM Bridge program. The fourth author was partially supported by a Simons postdoctoral fellowship from the Fields Institute for Research in Mathematical Sciences.

\smallskip
We have no conflict of interests to disclose.

\section{Preliminaries}\label{secprelim}
We work over the field of complex numbers. 
One of the central tools for deforming a finite morphism to a morphism of smaller degree is to construct a suitable multiple structure, called rope, on the image of the morphism.
\begin{definition}\phantomsection\label{defropes}
Let $Y$ be a reduced connected scheme and let $\mathcal{E}$ be a vector bundle of rank $m-1$ on $Y$. A {\it rope of multiplicity $m$ on $Y$ with conormal bundle $\mathcal{E}$} is a scheme $Y'$ with $Y'_{\textrm{red}}=Y$ such that
\begin{itemize}
    \item[(1)] $\mathcal{I}_{Y/Y'}^2=0$,
    \item[(2)] $\mathcal{I}_{Y/Y'}=\mathcal{E}$ as $\mathcal{O}_Y$ modules.
\end{itemize}
If $\mathcal{E}$ is a line bundle then $Y'$ is called a {\it ribbon} on $Y$.
\end{definition}

{Let $\varphi:X\to\mathbb{P}^N$ be a morphism} from a smooth, projective variety \textcolor{black}{$X$,} which is finite onto a smooth variety $Y\hookrightarrow\mathbb{P}^N$. \textcolor{black}{Recall that} {the space} $H^0(\mathcal{N}_{\varphi})$, \textcolor{black}{where
$\mathcal{N}_{\varphi}$ is the normal sheaf of $\varphi$,} parametrizes the 
infinitesimal deformations of $\varphi$
\textcolor{black}{(see, e.g., \cite[Section 3.4.2]{Ser})}. Suppose $\mathcal{E}$ is the trace zero module of the induced morphism $\pi:X\to Y$. It is shown in \cite[Proposition 2.1]{Go}, that the space  $H^0(\mathcal{N}_{Y/\mathbb{P}^N}\otimes\mathcal{E})$ parametrizes the pairs  $(\widetilde{Y},\widetilde{i})$ where $\widetilde{Y}$ is a rope on $Y$ with conormal bundle $\mathcal{E}$ and $\widetilde{i}:\widetilde{Y}\to\mathbb{P}^N$ is a morphism that extends $i$. The relation between these two cohomology groups is given by the following proposition.
\begin{proposition}\label{propngonzalez}
 (\cite[Proposition 3.7]{Go}) Let $X$ be a smooth variety and let $\varphi: X \to \mathbb{P}^N$ be a morphism that factors as $\varphi = i \circ \pi$, where $\pi$ is a finite cover of a smooth variety $Y$
and $i:Y\hookrightarrow\mathbb{P}^N$ is an embedding. Let $\mathscr{E}$ be the trace zero module of $\pi$ and let $\mathscr{I}$ be the ideal sheaf of $i(Y)$. There exists a homomorphism
\begin{equation*}
    H^0(\mathscr{N}_{\varphi})\xrightarrow{\psi}\textrm{Hom}(\pi^*(\mathscr{I}/\mathscr{I}^2),\mathscr{O}_X)
\end{equation*}
that appears when taking cohomology on the commutative diagram \cite[(3.3.2)]{Go}. Since
\begin{equation*}
    \textrm{Hom}(\pi^*(\mathscr{I}/\mathscr{I}^2),\mathscr{O}_X)=H^0(\mathscr{N}_{Y/\mathbb{P}^N})\oplus H^0(\mathscr{N}_{Y/\mathbb{P}^N}\otimes\mathscr{E}),
\end{equation*}
the homomorphism $\psi$ has two components;
\begin{equation*}
    H^0(\mathscr{N}_{\varphi})\xrightarrow{\psi_1}H^0(\mathscr{N}_{Y/\mathbb{P}^N})\textrm{ and }H^0(\mathscr{N}_{\varphi})\xrightarrow{\psi_2}H^0(\mathscr{N}_{Y/\mathbb{P}^N}\otimes\mathscr{E}).
\end{equation*}
\end{proposition}

Throughout this article, we will use the following result from \cite{GGP2}. Note the misprint in \cite{GGP2}, where the word "generically" in the last line of the statement is missing.

\begin{theorem} (\cite[Theorem 1.4]{GGP2})\label{main}
Let $X$ be a smooth projective variety and let $\varphi:X\to \mathbb{P}^N$ be a morphism that factors through an embedding $Y\hookrightarrow\mathbb{P}^N$ with $Y$ smooth. Let $\pi:X\to Y$ be the induced morphism which we assume to be finite of degree $n\geq 2$. Let $\widetilde{\varphi}:\widetilde{X}\to\mathbb{P}^N_{\Delta}$ $(\Delta=\textrm{Spec}\left(\frac{\mathbb{C}[\epsilon]}{\epsilon^2}\right))$ be a locally trivial first order infinitesimal deformation of $\varphi$ and let $\nu\in H^0(\mathscr{N}_{\varphi})$ be the class of             
 $\widetilde{\varphi}$.  If
\begin{itemize}
    \item[(a)] the homomorphism $\psi_2(\nu)$ has rank $k> \frac{n}{2}-1$, and 
    \item[(b)] there exists an algebraic formally semiuniversal deformation of ${\varphi}$ and ${\varphi}$ is unobstructed,
\end{itemize}
then there exists a flat family of morphisms, $\Phi: \mathscr{X}\to\mathbb{P}^N_T$ over $T$, where
$T$ is a smooth irreducible algebraic curve with a distinguished point $0$, such that
\begin{enumerate}
    \item[(1)] $\mathscr{X}_t$ is a smooth irreducible projective variety,
    \item[(2)] the restriction of $\Phi$ to the first order infinitesimal neighborhood of $0$ is $\widetilde{\varphi}$, and 
    \item[(3)] for $t\neq 0$, $\Phi_t$ is finite and generically one-to-one onto its image in $\mathbb{P}^N_t$.
\end{enumerate}
\end{theorem}

The first and the second author, along with Gonz\'alez, in fact gave a criterion under which a finite morphism deforms to an embedding. 
\begin{theorem} (\cite[Theorem 1.5]{GGP1})\label{main1}
Under the assumption of Theorem ~\ref{main}, suppose moreover $\psi_2(\nu)$ is a surjective homomorphism. Then there exists a flat family of morphisms, $\Phi: \mathscr{X}\to\mathbb{P}^N_T$ over $T$, where
$T$ is a smooth irreducible algebraic curve with a distinguished point $0$, such that
\begin{enumerate}
    \item[(1)] $\mathscr{X}_t$ is a smooth irreducible projective variety,
    \item[(2)] the restriction of $\Phi$ to the first order infinitesimal neighborhood of $0$ is $\widetilde{\varphi}$, and 
    \item[(3)] for $t\neq 0$, $\Phi_t$ is a closed immersion into $\mathbb{P}^N_t$.
\end{enumerate}
\end{theorem}

The authors of \cite{GGP1} in fact showed that under the assumptions above, (Im$\Phi)_0$ is an embedded rope $\widetilde{Y}$ corresponding to $\psi_2(\nu)$. It is clear from the above theorems that we need to produce homomorphisms of appropriate rank between the vector bundles $\mathcal{I}/\mathcal{I}^2$ and $\mathcal{E}$. In order to do that, we need a theorem of B\u{a}nic\u{a}. We need the following definition in order to state the theorem. 

\begin{definition}
Let $\nu:\mathcal{E}\to \mathcal{F}$ be a morphism  between vector bundles of rank $e$ and $f$ respectively, on an irreducible complex projective  variety $X$. For any positive integer $k\leq \textrm{min}\{e,f\}$, we define the {\it k-th degeneracy locus} $D_k(\nu)$ as the subscheme cut out by the minors of order $k+1$ of the matrix locally representing $\nu$.
\end{definition}

Now we state the result of B\u{a}nic\u{a}. We will use this result to deform a finite morphism to a birational morphism.
\begin{theorem}\label{banica}
(\cite[\S 4.1]{B}) Let $\mathcal{E}$ and $\mathcal{F}$ be vector bundles on a projective variety $X$ of rank $e$ and $f$ respectively. Assume $\mathcal{E}^*\otimes \mathcal{F}$ is globally generated. Then, for a general morphism $\nu:\mathcal{E}\to \mathcal{F}$, the subschemes $D_k(\nu)$ either are empty or have pure codimension $(e-k)(f-k)$ in $X$ and the singular locus $\textrm{Sing}(D_k(\nu))=D_{k-1}(\nu)$.
\end{theorem}

The double covers of complete intersections are studied in \cite{BG20} in  more detail. In case of double covers, there are two possibilities, namely the general deformation is either birational, or it is finite of degree $2$. However, for higher degree covers, the degree of \textcolor{black}{a} general deformation might drop, but it might not be birational onto its image. We will show that this case indeed appears. In order to prove that, the following proposition  \textcolor{black}{and Proposition~\ref{jayan} will be} crucial.

The following result is the consequence of \cite[Proposition 1.10]{Weh}.
\begin{proposition} \label{corweh}
Let $\pi:X\to Y$ be a finite, flat  morphism between smooth projective varieties with trace zero module $\mathscr{E}$ and let $\psi:Y\to Z$ be a non--degenerate morphism to a smooth variety $Z$. Let $\varphi=\psi\circ\pi$ be the composed morphism. If $H^0(\mathscr{N}_{\psi}\otimes\mathscr{E})=0$ then the natural map between the functors ${\bf Def}({\pi/Z})\to{\bf Def}_{\varphi}$ is smooth.
\end{proposition}
\noindent\textit{Proof.} We apply \cite[Proposition 1.10]{Weh} to the following commutative diagram.
\begin{equation*}\label{eq.comm.square.triangle}
		\xymatrix@C-20pt@R-12pt{
	X  \ar[rr]^{\pi} 
	\ar[dr]_{\varphi} 
	&& Y \ar[dl]^{\psi} \\
	& Z}
\end{equation*}
The maps $\beta_1$ and $\beta_2$ of \cite[Proposition 1.10]{Weh}
become the following: $$\beta_1: H^0(\mathcal{N}_{\psi}) \to H^0(\pi^*\mathcal{N}_{\psi})\textrm{, and }
\beta_2: H^1(\mathcal{N}_{\psi}) \to H^1(\pi^*\mathcal{N}_{\psi}).$$ The assertion follows since the map $\beta_2$ is always injective and $\beta_1$  is surjective if $H^0(\mathcal{N}_{\psi}\otimes\mathcal{E})=0$.\QEDB\par

\section{Deformation \textcolor{black}{theory} {of} \textcolor{black}{finite} covers of complete intersections}\label{3}
  
  \color{black}
  In this section we develop the deformation theory of 
  finite covers of complete intersection subvarieties of projective space. We will use our theory to construct in a systematic way other subvarieties, of dimension $m \geq 3$, of projective spaces of any dimension. These new subvarieties cover all possibilities in terms of dimension and codimension and infinitely many of them are non--complete intersections. 
  
  \smallskip
  
  We will require  our covers to satisfy two additional conditions (see (2.1) and (2.2) of Set-up~\ref{setup1} below), namely, the splitting of its trace zero module as a direct sum of line bundles and the vanishing of the first cohomology group of its normal sheaf. As Subsection~\ref{subsection.abelian.dihedral}
  makes clear, abelian and simple dihedral covers of complete intersections satisfy these two conditions, so there are plenty of covers to which our theory applies. This will be our set-up: 
  \color{black}

\begin{set-up}\label{setup1} Throughout  \textcolor{black}{the  remaining of the} article, unless otherwise stated, we will use the following set-up.
\begin{itemize}[leftmargin=0.4in]
    \item[(1)] Let $i:Y\hookrightarrow\mathbb{P}^N$ be a smooth, 
    \textcolor{black}{non--degenerate,} complete intersection 
    \textcolor{black}{subvariety} of multidegree $\underline{d}=(d_1,d_2,\cdots,d_r)$, with $r\geq 1$ and $2\leq d_1\leq d_2\leq \cdots \leq d_r$. \textcolor{black}{Let $m$ be the dimension of $Y$ and assume $m \geq 3$.} 
    Set
    \begin{equation*}
        \delta =\sum\limits_{i=1}^{r}d_i\textrm{ and } d=\prod\limits_{i=1}^r d_i.
    \end{equation*}
    \item[(2)] 
    \textcolor{black}{Let 
    $X$ be a smooth, irreducible variety,} let $\pi:X\to Y$ be a finite morphism \textcolor{black}{of degree $n$, $n \geq 2$ and let 
    $\varphi=i\circ\pi$}. 
    Let $\pi_*\mathcal{O}_X=\mathcal{O}_Y\oplus \mathcal{E}$, 
    \textcolor{black}{where $\mathcal E$ is the trace zero module of $\pi$.} 
    \textcolor{black}{Let} 
    $\mathcal N_\pi$ be the normal sheaf of $\pi$. 
    \color{black}{Assume furthermore that}
    \smallskip
    
    \begin{itemize}[leftmargin=0.6in]
        \item[(2.1)] $\mathcal E$ splits as direct sum of line bundles; 
        \item[(2.2)] $H^1(\mathcal N_\pi)=0$.
    \end{itemize}
    
    \smallskip
  Finally,  let 
  $\mathcal{O}_Y(k):=i^*\mathcal{O}_{\mathbb{P}^N}(k)$ for any 
  $k \in \mathbb Z$ and let 
    \color{black}
    \begin{equation*}
        \mathcal{E}=\bigoplus\limits_{i=1}^{n-1}\mathcal{O}_Y(-k_i),\textrm{ i.e., }\pi_*\mathcal{O}_X=\mathcal{O}_Y\oplus\left(\bigoplus\limits_{i=1}^{n-1}\mathcal{O}_Y(-k_i)\right),
    \end{equation*}
    \textcolor{black}{(see Remark~\ref{int} (1) below)}, 
    \textcolor{black}{where} $k_1\leq k_2\leq\cdots\leq k_{n-1}$ \textcolor{black}{are, necessarily positive, integers.} 
\end{itemize}
\end{set-up}

\begin{remark}\label{int}
We make a note of the following \textcolor{black}{well known} facts that we will use without explicitly stating.
\begin{itemize}
    \item[(1)] Since $Y$ is a complete intersection in $\mathbb{P}^N$, it follows from the Grothendieck–Lefschetz theorem (recall that 
    \textcolor{black}{$m\geq 3$}) that any line bundle on $Y$ is the restriction of a line bundle on $\mathbb{P}^N$. \textcolor{black}{In particular,} 
    $\textrm{Pic}(Y) = \mathbb{Z}$. 
    \item[(2)] Since $Y$ is a complete intersection of multidegree $\underline{d}=(d_1,\cdots, d_r)$, it follows that $$\mathcal{N}_{Y/\mathbb{P}^N}=\bigoplus\limits_{i=1}^r\mathcal{O}_Y(d_i).$$
    \item[(3)] For $a\in\mathbb{Z}$, $H^0(\mathcal{O}_Y(a))\neq 0$ $\iff$ $a\geq 0$ $\iff$ $\mathcal{O}_Y(a)$ is globally generated.
    \item[(4)] For $a\in\mathbb{Z}$, $H^i(\mathcal{O}_Y(a))= 0$, for all $1\leq i\leq \textcolor{black}{m-1}$. 
    \color{black}
    \item[(5)] Let $\mathcal I$ be the ideal sheaf of $i(Y)$ inside $\mathbb P^N$ and let  $H^i_*(\mathcal I)=\bigoplus_{\nu\in\mathbb{Z}} H^i(\mathcal I(\nu))$. For all $1 \leq i \leq m$,  
    $H^i_*(\mathcal I)=0$.
    \item[(6)] Because of (4), the variety $X$ is regular.
    \item[(7)] The morphism 
    \color{black}
    $\varphi$ is induced by the complete linear series $|\pi^*\mathcal{O}_Y(1)|$ if and only if $k_1\geq 2$.
\end{itemize}
\end{remark}

\color{black}
We now state and prove our main results on deformations of finite covers of complete intersections. First we need the following lemma.

\color{black}
\begin{lemma}\label{unobs}
In the situation of Set-up ~\ref{setup1}, $\varphi$ has an algebraic formally semiuniversal deformation space and it is unobstructed. Moreover, 
the map (see Proposition ~\ref{propngonzalez}) $$\psi_2: H^0(\mathcal{N}_{\varphi})\to H^0(\mathcal{N}_{Y/\mathbb{P}^N}\otimes\mathcal{E})$$ is surjective, where $\mathcal{E}$ is the trace zero module of $\pi$.
\end{lemma}
\noindent\textit{Proof.} Notice that $H^2(\mathcal{O}_X)=H^2(\mathcal{O}_Y)\oplus H^2(\mathcal{E})$. Consequently, by \cite[Proposition 1.7]{BGG20}, the morphism $\varphi$ has an algebraic formally semiuniversal deformation space since by Remark ~\ref{int}, (4) $H^2(\mathcal{O}_X)=0$. It follows from 
our assumption that $H^1(\mathcal{N}_{\pi})=0$. In particular, $\psi_2$ surjects due to the following exact sequence (see \cite[Lemma 3.3]{Go}) and projection formula:
\begin{equation}\label{gonex}
    0\to\mathcal{N}_{\pi}\to\mathcal{N}_{\varphi}\to\pi^*\mathcal{N}_{Y/\mathbb{P}^N}\to 0. 
\end{equation}
 Moreover, by Remark ~\ref{int}, (4), we get that $H^1(\pi^*\mathcal{N}_{Y/\mathbb{P}^N})=H^1(\mathcal{N}_{Y/\mathbb{P}^N})\oplus H^1(\mathcal{N}_{Y/\mathbb{P}^N}\otimes\mathcal{E})=0$. Then, from the short exact sequence ~\eqref{gonex}
we get $H^1(\mathcal{N}_{\varphi})=0$.\QEDB\par 
\vspace{5pt}

The following proposition gives a criterion under which the degree of a finite morphism remains unchanged for any of its deformation.

\begin{proposition}\label{nonexistence}
\color{black} Let   Set-up ~\ref{setup1} hold except maybe hypothesis (2.2).
\color{black} If $d_r<k_1$ then any deformation of $\varphi$ is finite of degree $n$ onto its image, which is a smooth, complete intersection subvariety of $\mathbb{P}^N$ of multidegree $\underline{d}=(d_1, d_2, \cdots , d_r)$.
\end{proposition}
\noindent\textit{Proof.} By hypothesis, and Remark ~\ref{int}, (2), we obtain $H^0(\mathcal{N}_{Y/\mathbb{P}^N}\otimes\mathcal{E})=0$. The conclusion follows from Proposition ~\ref{corweh} and \cite{Ser75}.\QEDB\par 
\vspace{5pt}

Our methods of deformations of the finite morphism to an embedding is intimately related to the existence of some special rope structure on $Y$. We construct ropes of suitable ranks on $Y$ and we smooth them. This in turn deforms the finite morphism to an embedding or a birational map. One can explicitly construct smoothing family of split multiple structures, but our constructions are general.

Now we state our main \textcolor{black}{theorems} 
that give \textcolor{black}{criteria to determine when the general deformation of the covering morphism is birational (see Theorem~\ref{birational}) and when is not only birational but also an embedding  (see Theorem~\ref{theorem.embedding}).} 
Our resuls are a consequence of the existence of appropriate ropes.

\begin{theorem}\label{birational}
In the situation of Set-up ~\ref{setup1}, let
 $r>\floor{\frac{n}{2}}-1$ and $d_{r-\floor{\frac{n}{2}}+1}\geq k_{n-1}$. Then a general element of the algebraic formally semiuniversal deformation space of $\varphi$ is finite and birational onto its image.
\end{theorem}

\begin{remark}
Observe that $2r+2-n-N=r-(N-r)-(n-2)<r$ as $n\geq 2$ and $N-r\geq 3$.
\end{remark}

\noindent\textit{Proof of Theorem ~\ref{birational}.} 
Let $\mathcal{N}'=\mathcal{O}_Y(d_a)\oplus\mathcal{O}_Y(d_{a+1})\oplus\cdots\oplus\mathcal{O}_Y(d_r)$ where $a=r-\floor{\frac{n}{2}}+1$ and let $\mathcal{E}$ be the trace zero module of $\pi$. Notice that $\textrm{rank}(\mathcal{N}')=\floor{\frac{n}{2}}$. The condition $d_{r-\floor{\frac{n}{2}}+1}\geq k_{n-1}$ guarantees that $\mathcal{N}'\otimes\mathcal{E}$ is globally generated. It follows from Theorem ~\ref{banica} that, for a general homomorphism $\sigma:\mathcal{N}'^*\to\mathcal{E}$, $D_{\floor{\frac{n}{2}}-1}(\sigma)$ is a proper closed subvariety inside $Y$. Consequently, there exists a rank $\floor{\frac{n}{2}}$ homomorphism in $\textrm{Hom}(\mathcal{N}'^*,\mathcal{E})$, which in turn implies the existence of a rank $\floor{\frac{n}{2}}$ homomorphism in $\textrm{Hom}(\mathcal{N}_{Y/\mathbb{P}^N}^*,\mathcal{E})$. The assertion follows from Theorem ~\ref{main} and Lemma ~\ref{unobs}, since birationality is an open condition.
\QEDB\par

\begin{theorem}\label{theorem.embedding}
In the situation of Set-up ~\ref{setup1},
 let either of the following conditions (a) or (b) hold;
    \begin{itemize}
        \item[(a)] $r\geq n-1$ and there exists $1\leq l_1<l_2<\cdots<l_{n-1}$, such that $d_{l_j}=k_j \textrm{ for all } 1\leq j\leq n-1$;
         \item[(b)] $r>\frac{N+n-2}{2}$, and $d_{2r+2-n-N}\geq k_{n-1}$;
    \end{itemize}
    then there exist embedded ropes $\widetilde{Y}\hookrightarrow\mathbb{P}^N$ on $Y$ with conormal bundle $\mathcal{E}$, which are non--split if $\varphi$ is induced by a complete linear series. 
    
    \noindent
   For any embedded rope $\widetilde{Y}\hookrightarrow\mathbb{P}^N$ 
   on $Y$ with conormal bundle $\mathcal{E}$,  there exists a flat family $\Phi: \mathscr{X}\to\mathbb{P}^N_T$ over $T$, where $T$ is a smooth irreducible algebraic curve with a distinguished point $0$, such that
\begin{enumerate}
    \item[(I)] $\mathscr{X}_t$ is a smooth irreducible projective variety,
    \item[(II)] $\mathcal{X}_0=X$, $\Phi_0=\varphi$, and 
    \item[(III)] for $t\neq 0$, $\Phi_t$ is a closed immersion onto $\mathbb{P}^N_t$ and (Im$\Phi$)$_0$ is $\widetilde{Y}$.  
\end{enumerate} 
Consequently,  a general element of the algebraic formally semiuniversal deformation space of $\varphi$ is an embedding.
\end{theorem}

\noindent\textit{Proof.} We aim to show that if (a) or (b) holds, then there exists a surjective homomorphism in \linebreak $\textrm{Hom}(\mathcal{N}_{Y/\mathbb{P}^N}^*,\mathcal{E})$. Clearly that is the case if (a) holds. Now, assume (b) holds. As before, let $a=2r+2-n-N$ and set $\mathcal{N}''=\mathcal{O}_Y(d_a)\oplus\mathcal{O}_Y(d_{a+1})\oplus\cdots\oplus\mathcal{O}_Y(d_r)$. Notice that $\mathcal{N}''$ is a vector bundle of rank $(N-r)+(n-1)$ and the condition $d_{2r+2-n-N}\geq k_{n-1}$ guarantees that $\mathcal{N}''\otimes\mathcal{E}$ is globally generated. It follows from Theorem ~\ref{banica} that, for a general homomorphism $\sigma':\mathcal{N}''^*\to\mathcal{E}$, $D_{n-2}(\sigma')$ is empty. Indeed, otherwise it  has expected codimension $$((N-r)+(n-1)-(n-2))((n-1)-(n-2))=N-r+1>N-r$$ which is impossible. Thus, $\sigma'$ can be extended to a surjective homomorphism in $\textrm{Hom}(\mathcal{N}_{Y/\mathbb{P}^N}^*,\mathcal{E})$.\par 
Now, any surjective homomorphism 
of $\textrm{Hom}(\mathcal{N}_{Y/\mathbb{P}^N}^*,\mathcal{E})$ 
corresponds to
an embedded rope $\widetilde{Y}$ on $Y$, with conormal bundle $\mathcal{E}$.
 Assume $\varphi$ is induced by the complete linear series $|\pi^*\mathcal{O}_Y(1)|$ (equivalently, $k_1\geq 2$). Then any embedded rope $\widetilde{Y}\hookrightarrow\mathbb{P}^N$ on $Y$ with conormal bundle $\mathcal{E}$ is non--split. Indeed, it follows from the long exact sequence associated to the restricted Euler sequence twisted by $\mathcal{E}$:
\begin{equation*}
    0\to\mathcal{E}\to \mathcal{E}(1)^{\oplus N+1}\to T_{\mathbb{P}^N}\vert_Y\otimes\mathcal{E}\to 0
\end{equation*}
that $H^0(T_{\mathbb{P}^N}\vert_Y\otimes\mathcal{E})=H^1(T_{\mathbb{P}^N}\vert_Y\otimes\mathcal{E})=0$. Consequently, by the long exact sequence associated to the following exact sequence,
\begin{equation*}
    0\to T_Y\otimes\mathcal{E}\to T_{\mathbb{P}^N}\vert_Y\otimes\mathcal{E}\to \mathcal{N}_{Y/\mathbb{P}^N}\otimes\mathcal{E}\to 0,
\end{equation*}
it follows that $H^0(\mathcal{N}_{Y/\mathbb{P}^N}\otimes\mathcal{E})\cong H^1(T_Y\otimes\mathcal{E})$. Thus, the assertion follows from the fact that the class of a surjective homomorphism in $H^0(\mathcal{N}_{Y/\mathbb{P}^N}\otimes\mathcal{E})$ is nonzero.\par 
 The rest of the assertions of the statement of the theorem
 are consequences of 
 Theorem~\ref{main1}, Lemma ~\ref{unobs}, \cite[Theorem 2.2]{GGP1}, and the fact that being an embedding is an open condition.\QEDB\par

\vspace{5pt}

Next theorem gives a criterion under which the degree of a general deformed morphism is one half of the initial degree. 

\begin{theorem}\label{2:1}
In the situation of Set-up ~\ref{setup1}, assume $n$ is even, $n\geq 4$, and the following holds;
\begin{itemize}
    \item[(1)] $\pi=p_1\circ\pi_1$ where $\pi_1:X\to X_1$ is a
 \textcolor{black}{abelian} cover  of degree ${n}/{2}$ with trace zero module $$\mathcal{E}_2= p_1^*\left(\mathcal{O}_Y(-k_1')\oplus\cdots\oplus\mathcal{O}_Y(-k'_{\frac{n}{2}-1})\right), \ \ 
    k_1' \leq \cdots \leq k'_{\frac{n}{2}-1}$$ \textcolor{black}{and branch divisor $$ E = \Sigma_i p_1^*(D_i)$$ where $p_1^*D_i$ are the irreducible components of \textcolor{black}{$E$} and $D_i$ are divisors on $Y$} and $p_1:X_1\to Y$ is a double cover with trace zero module $\mathcal{E}_2=\mathcal{O}_Y(-l)$,
    \item[(2)] $k_1'>\textrm{max}\big\{2l,d_r\big\}$ and $d_r\geq l$.
\end{itemize}  
Then, a general element $\varphi'$ of the algebraic formally semiuniversal deformation space of $\varphi$ is \textcolor{black}{a morphism, which is} finite \textcolor{black}{and} of degree $\frac{n}{2}$ \textcolor{black}{onto its image}. If, moreover, one of the following holds;
\begin{itemize}
    \item[(1')] $d_s=l$ for some $1\leq s\leq r$, or
    \item[(2')] $r>\frac{N}{2}$, and $d_{2r-N}\geq l$,
\end{itemize}
then, a general element \textcolor{black}{$\varphi'$} of the algebraic formally semiuniversal deformation space of $\varphi$ is \textcolor{black}{a} finite \textcolor{black}{morphism} of degree $\frac{n}{2}$, onto a smooth variety.
\textcolor{black}{In particular, $\varphi'$ is flat.} \textcolor{black}{Moreover in this case the algebraic semiuniversal deformation spaces of both $\varphi$ and $X$ are smooth and uniruled.}
\end{theorem}

\noindent\textit{Proof.} We have the following short exact sequence where $\varphi_1=i\circ p_1$ (see \cite[Lemma 3.3]{Go});
\begin{equation}\label{ge}
    0\to\mathcal{N}_{p_1}\to\mathcal{N}_{\varphi_1}\to p_1^*\mathcal{N}_{Y/\mathbb{P}^N}\to 0.
\end{equation}
Note that $h^0(\mathcal{N}_{p_1}\otimes \mathcal{E}_2)=\sum\limits_{i=1}^{\frac{n}{2}-1}h^0(\mathcal{O}_B(2l-k_i'))$, where $B\in|\mathcal{O}_Y(2l)|$ is the branch divisor of $p_1$. It is easy to see from the following exact sequence;
$$0\to\mathcal{O}_Y(-k_i')\to \mathcal{O}_Y(2l-k_i')\to \mathcal{O}_B(2l-k_i') \textcolor{black}{\longrightarrow 0},$$
that $h^0(\mathcal{N}_{p_1}\otimes \mathcal{E}_2)=\sum\limits_{i=1}^{\frac{n}{2}-1}h^0(\mathcal{O}_Y(2l-k_i'))$. By assumption, $k_1'>2l$, and consequently, $k_i'>2l$, for all $i$. Thus, $h^0(\mathcal{N}_{p_1}\otimes \mathcal{E}_2)=0$. Also, since $k_j'>d_i$, we obtain $$h^0(p_1^*(\mathcal{N}_{Y/\mathbb{P}^N}\otimes\mathcal{E}_2))=\sum\limits_{i=1}^{r}\sum\limits_{j=1}^{\frac{n}{2}-1}h^0(\mathcal{O}_Y(d_i-k_j'))+\sum\limits_{i=1}^{r}\sum\limits_{j=1}^{\frac{n}{2}-1}h^0(\mathcal{O}_Y(d_i-k_j'-l))=0.$$ Consequently, by tensoring the exact sequence ~\eqref{ge} by $\mathcal{E}_2$ and taking the long exact sequence of cohomology, one finds that $h^0(\mathcal{N}_{\varphi_1}\otimes \mathcal{E}_2)=0$. It follows from Proposition ~\ref{corweh} that any deformation of $\varphi$ is of degree greater than or equal to $\frac{n}{2}$, since it factors through a deformation of $\pi_1$.\par 
Since $\varphi$ is unobstructed (Lemma ~\ref{unobs}), $h^1(\mathcal{N}_{p_1})=0$ and $h^0(\mathcal{N}_{Y/\mathbb{P}^N}\otimes \mathcal{O}_Y(-l))\neq 0$, it follows that there exists a deformation $\Psi:\mathcal{X}_1\to\mathbb{P}_T^N$ over a smooth pointed algebraic curve $(T,0)$ for which $\Psi_0=\varphi_1$ and $\Psi_t$ is of degree $1$ for $t\neq 0$. \textcolor{black}{Now notice that $H^i(\mathscr{O}_{X_1}) = H^i(\mathscr{O}_Y) \bigoplus H^i(\mathscr{O}_Y(-l)) = 0$ for \textcolor{black}{i = $1$, $2$} since $Y$ is at least a threefold and the intermediate cohomology vanishes.  Also we have $H^1(p_1^*D_i) = H^1(D_i) \bigoplus H^1(D_i \otimes \mathscr{O}_Y(-l)) = 0$ for the same reason.} It follows from \textcolor{black}{Proposition} ~\ref{jayan}, that, possibly after shrinking $T$, there exists $\Pi:\mathcal{X}\to\mathcal{X}_1$ such that $\Pi_0=\pi_1$. Thus, $\Phi=\Psi\circ\Pi:\mathcal{X}\to\mathbb{P}^N_T$ is a deformation of $\Phi_0=\varphi$ such that $\Phi_t$ of degree $\frac{n}{2}$ for all $t\neq 0$. Thus, a general element of the algebraic formally semiuniversal deformation space of $\varphi$ will have degree less than or equal to $\frac{n}{2}$. The conclusion follows since we have already showed that any deformation of $\varphi$ is of degree greater than or equal to  $\frac{n}{2}$.\par 
The last assertion is clear from Theorem ~\ref{theorem.embedding}, since under the assumption, $\varphi_1$ deforms to an embedding, and being an embedding is an open condition. \par
\textcolor{black}{The deformation space of $\varphi$ is fibered over the deformation space of $\varphi_1$ by a product of projective spaces and is hence uniruled. Since $\varphi_1$ is unobstructed and $H^1(N_{\varphi_1}) \to H^1(T_{X_1})$ surjects, we have that the deformation space of $X$ is fibered over the deformation space of $X_1$ by a product of projective spaces and is hence smooth and uniruled.} That completes the proof.\QEDB\par

\color{black}
\begin{remark}
Let $Y'$ be the image of $\varphi'$.
Note that, if, in addition to (1) and (2), hypothesis (1') and (2') in Theorem~\ref{2:1} are assumed, then,  the length of the fiber \textcolor{black}{of}  $\varphi'$, \textcolor{black}{over any closed point of $Y'$,}   is $n/2$. 
If (1') and (2') are not assumed, then $Y'$ might be singular and one can only guarantee that the length of the fiber \textcolor{black}{of}  $\varphi'$, 
\textcolor{black}{over any closed point of a dense open set of $Y'$,}   is $n/2$. 
\end{remark}

\color{black}
\begin{remark}\label{cls}
In the situation of Set-up ~\ref{setup1}, assume that 
\textcolor{black}{$k_1\geq 2$}. This is equivalent to  
$\varphi$ being induced by the complete linear series $|\varphi^*\mathcal{O}_{\mathbb{P}^N}(1)|$. Then, for any deformation $\Phi:\mathcal{X}\to\mathbb{P}^N_T$ of $\varphi$ over a smooth \textcolor{black}{variety} $(T,0)$ \textcolor{black}{(in particular, for those
in Proposition~\ref{nonexistence} and, \textcolor{black}{provided $k_1 \geq 2$,} Theorems~\ref{birational}, \ref{theorem.embedding} and \ref{2:1})}, we may assume (possibly after shrinking $T$), that $\Phi_t$ is induced by a complete linear series. This is a consequence of semicontinuity and the fact that being a non--degenerate morphism (in the sense that the image is not contained in any hyperplane) is an open condition.
\textcolor{black}{In particular, if $\varphi'$ corresponds to a general element of the algebraic formally semiuniversal deformation space of $\varphi$, then $\varphi'$ is  induced by a complete linear series. In particular, if $k_i\geq 2$ for all $i$, then the subvarieties $\Phi_t(\mathcal X_t)$ of Theorem~\ref{theorem.embedding} are embedded in $\mathbb P^N$ by complete linear series.}
\end{remark}

We end this section with a Barth-type result for covers of complete intersections, with a flavor similar to
\cite[Proposition 3.1]{L} for covers of projective space. Note that, unlike there, our result does not depend on the degree of the cover.

\begin{corollary}
In the situation of Set-up ~\ref{setup1}, assume the hypothesis of Theorem ~\ref{theorem.embedding} (a) is satisfied and $N\geq 2r+2$. Then Pic$(X)=\mathbb{Z}$. 
\end{corollary}

\noindent\textit{Proof.} 
Let $(T,0)$ be \textcolor{black}{a smooth} algebraic curve and let $\Phi:\mathcal{X}\to\mathbb{P}_t^N$ be a deformation for which $\Phi_t$ is an embedding for all $t\neq 0$. Since $H^2(\mathcal{O}_X)=0$, it follows that Pic$(X)\hookrightarrow\textrm{Pic}(\mathcal{X}_t)$.
By the theorem of 
\textcolor{black}{Barth-Larsen (see 
\cite{Larsen}, see also
\cite[Corollary 3.2.3]{positivityI})}, we know that Pic$(\mathcal{X}_t)=\mathbb{Z}$. The conclusion follows from the projectivity of $X$.\QEDB

\section{Abelian and dihedral covers}\label{secabel}

Theorems~\ref{birational} and \ref{theorem.embedding} work for 
finite covers of complete intersections of dimension 
$m \geq 3$ that satisfy conditions (2.1) and (2.2) of Set-up~\ref{setup1}. \textcolor{black}{Theorem~\ref{2:1} also works for 
finite covers of complete intersections  satisfying conditions (2.1) and (2.2) of Set-up~\ref{setup1} and possessing an intermediate cover of certain kind.} We see now that these conditions are satisfied by  large classes of covers of complete intersection subvarieties, namely, abelian covers  and
simple dihedral covers.  \textcolor{black}{Now we give the details of the construction of both kinds of covers. In this section, $Y$ is a smooth variety which is \emph{not necessarily a complete intersection}.} 

\subsection{\textcolor{black}{Abelian covers}}
\label{subsection.abelian.dihedral}

We recall the definition of an abelian cover. 

\color{black}

\begin{definition}\label{defgalois}
 Let \textcolor{black}{$Y$} be a variety and let $G$ be a finite group. A {\it Galois cover of \textcolor{black}{$Y$} with Galois group $G$} is a finite flat morphism
 \textcolor{black}{$\pi:X \to Y$} together with a faithful action of $G$ on $X$ that exhibits $Y$ as a quotient of \textcolor{black}{$X$} via $G$. We call a Galois cover {\it abelian}, if $G$ is abelian.
\end{definition}

\color{black}
\begin{remark}
\textcolor{black}{If $\pi: X \to Y$ is an abelian cover, then}
 the vector bundle $\pi_*\mathcal O_X$ 
 splits as a direct sum of line bundles on $Y$. 
 \color{black}
 More precisely 
 (see e.g. \cite[(1.1)]{Pa})
 \begin{equation}\label{formula.O_X.splitting}
 \pi_*\mathscr{O}_X=\bigoplus\limits_{\chi\in G^*}L_{\chi}^{-1},
 \end{equation} 
 where \textcolor{black}{$G^*$ is the character group of $G$,} $L_\chi$ is a line bundle on $Y$, $G$ acts on $L_\chi$ via 
 the character $\chi$ and the invariant summand $L_1$ is isomorphic to $\mathcal O_Y$. \textcolor{black}{In particular, an abelian cover of a complete intersection subvariety  satisfies (2.1) of Set-up~\ref{setup1}.} 
\end{remark}

 \color{black}
 \smallskip
 \textcolor{black}{In order to check abelian covers satisfy (2.2) of Set-up~\ref{setup1}} we need to know the structure of $\pi_*\mathcal N_\pi$. For this we need to introduce some notation (see \cite{Pa} for 
 further details).
 Let $D$ be the branch divisor of $\pi$. Let $\mathscr{C}$ be the set of cyclic subgroups of $G$ and for all $H\in\mathscr{C}$, denote by $S_H$ the set of generators of the group of characters $H^*$. Then, we may write $$D=\sum\limits_{H\in\mathscr{C}}\sum\limits_{\psi\in S_H}D_{H,\psi}$$ where $D_{H,\psi}$ is the sum of all the components of $D$ that have inertia group $H$ and character $\psi$. For every $\chi\in G^*$, and $H\in \mathscr{C}$, and for every $\psi\in S_H$, one may write $\chi|_H=\psi^{i_{\chi}}$, $i_{\chi}\in\{0,\cdots,m_H-1\}$, where $m_H$ is the order of $H$.
 For every character $\chi\in G^*$, let 
 \begin{equation*}
     {S_{\chi}=\{(H,\psi):\chi|_H\neq \psi^{m_H-1}\}.}
 \end{equation*}

\smallskip

\begin{proposition}\label{pushnpi} (\cite[Corollary 4.1]{Pa}) Let $\pi:X \to Y$ be an abelian cover with Galois group $G$, with $X$ and $Y$ smooth. Assume the branch divisor $D$ of $\pi$ is normal crossing. Then,  
 $$(\pi_*\mathcal{N}_{\pi})^{\chi} \cong \bigoplus_{(H,\psi) \in S_{\chi}} \mathcal{O}_{D_{H,\psi}}(D_{H,\psi}) \otimes L_{\chi}^{-1}.$$
\end{proposition}

  \color{black}
\begin{corollary} 
Let $\pi: X \longrightarrow Y$ be as in Set-up~\ref{setup1}. 
If $\pi$ is abelian, then (2.2) of Set-up~\ref{setup1} is satisfied, i.e., $H^1(\mathcal{N}_{\pi})=0$.
\end{corollary}

\noindent\textit{Proof.} 
If follows from Proposition~\ref{pushnpi} and Remark ~\ref{int}, (4).\QEDB \par
\smallskip

\textcolor{black}{We end this subsection by proving a result on deformations of abelian covers:} 
\begin{proposition}\label{jayan}
Let $\pi: X \to Y$ be an abelian cover with $X$ and $Y$ smooth, 
projective varieties with \textcolor{black}{$H^1(\mathscr{O}_Y) = $} $H^2(\mathscr{O}_Y) = H^1(D_i) = 0$ for all irreducible components $D_i$ of the branch divisor $D$. Suppose $p: \mathscr{Y} \to T$ be a deformation of $Y$ ($p$ is proper, flat and surjective) over a smooth \textcolor{black}{algebraic} curve $T$. Then there exists a deformation $\Pi: \mathscr{X} \to \mathscr{Y} \to T$ of $\pi$ over $T$.
\end{proposition}

\color{black}

\begin{proof}
By \cite[Theorem $2.1$]{Pa}, if $G$ is the  \textcolor{black}{Galois} group \textcolor {black}{of $\pi$}, then \textcolor{black}{$\pi$} is completely determined by \textcolor{black}{the} line bundles $L_{\chi},$ \textcolor{black}{the divisors $D_{H, \psi}$ above,} corresponding to a choice of a cyclic subgroup $H \subseteq G$ and a generator $\psi \in H^*$, and \textcolor{black}{linear equivalent relations} of the form
\begin{equation}\label{eq.abelian.cover.relations}
    L_{\chi}+L_{\chi'} = L_{\chi\chi'}+\Sigma_{H}\Sigma_{\psi} \epsilon_{\chi, \chi'}^{H,\psi} D_{H,\psi}.
\end{equation}
Now note that, since $H^2(\mathscr{O}_Y)=0$, \textcolor{black}{by \cite[Propositions 2.2.5 (iv), 2.3.6]{Ser},} all the line bundles $L_{\chi}$ lift \textcolor{black}{to} $\mathscr{Y} \to T$. Also since $H^1(D_i) = 0$, \textcolor{black}{by  \cite[Remark 2.10]{moduli},} every $D_{H,\psi}$ lifts to a divisor on $\mathscr{Y} \to T$. Now, since  \textcolor{black}{\eqref{eq.abelian.cover.relations}} is satisfied in the central fibre \textcolor{black}{$\mathscr{Y}$}, and $Y$ is regular, the same linear equivalence relation holds \textcolor{black}{on $\mathscr{Y}_t$}, for a general $t \in T$ and, once again, by \cite[Theorem $2.1$]{Pa} we have a family of abelian covers $\mathscr{X} \to \mathscr{Y} \to T$ \textcolor{black}{extending $\pi$}.
\end{proof}

\color{black}

\subsection{Simple dihedral covers} We will see now that conditions
(2.1) and (2.2) of 
our set-up~\ref{setup1} also hold for
simple dihedral covers. \color{black}
Dihedral covers were studied by Catanese and Perroni in \cite{CP}. We denote the dihedral group of order $2n'$ by $D_{n'}$ and we recall the following theorem.

\smallskip

\color{black}

\begin{theorem}\label{cp}
(\cite[Theorem 6.1]{CP}). Let $Y$ be a smooth variety and let $n'\geq 3$ be an integer. Let $L\to Y$ be a line bundle with $s_1\in H^0(L^{\otimes n'})$ and $s_2\in H^0(L^{\otimes 2})$ such that
\begin{itemize}
    \item[(1)] The zero locus of $s_1^2-s_2^{n'}$ is smooth in the open set $s_2\neq 0$, and 
    \item[(2)] The divisors $\{s_1=0\}$ and $\{s_2=0\}$ intersect transversally.
\end{itemize}
Define $X\subset L\oplus L$ to be the variety defined by $(u,v)\in L\oplus L\oplus L$ that satisfies the following equations:
\begin{equation*}
    uv=s_2,\quad u^{n'}-2s_1v+s_2^{n'}=0.
\end{equation*}
Then the restriction to $X$ of the fiber bundle projection $L\oplus L\to Y$ is a  \textcolor{black}{Galois} cover with \textcolor{black}{group $D_{n'}$ and} branch divisor $\{s_2^{n'}-s_1^2=0\}$. Furthermore, if $\{s_1=0\}\cap \{s_2=0\}\neq\phi$, then $X$ is irreducible.
\end{theorem}

It was shown in \cite[Proposition 6.3]{CP} that the condition (2) in Theorem ~\ref{cp} is satisfied for a general choice of $s_1$ and $s_2$. Furthermore, it is easy to see that the condition (2) of the theorem and the irreducibility of $X$ can be simultaneously guaranteed if $L$ is ample and globally generated.

\begin{definition}
A {\it simple $D_{n'}$ cover} of $Y$ is the \textcolor{black}{Galois} cover $\pi:X\to Y$ \textcolor{black}{with group $D_{n'}$}
given as in Theorem ~\ref{cp} by the restriction to $X$ of the fiber bundle
projection $L\oplus L\to Y$.
\end{definition}
Note that a simple $D_{n'}$ cover $\pi$ has degree $n=2n'$. Given a simple $D_{n'}$ cover \textcolor{black}{$\pi$}, the following formulas for the push-forward of the structure sheaf and for the  canonical bundle hold (\cite[\S 6.1]{CP}):
\begin{equation}\label{pushdn}
    \pi_*\mathcal{O}_X=\bigoplus\limits_{i=0}^{n'-1}[\mathcal{O}_Y(-iL)\oplus\mathcal{O}_Y(-(n'-i)L)],\quad K_X=\pi^*(K_Y(n'L)).
\end{equation}
In particular, simple dihedral covers satisfy (2.1) of Set-up~\ref{setup1}.

\begin{theorem}\label{sedn}
Let $\pi: X \to Y$ be a simple \textcolor{black}{$D_{n'}$} cover corresponding to a line bundle $L$ and sections $s_1 \in H^0(L^{\otimes n'})$ and $s_2 \in H^0(L^{\otimes 2})$. Then we have the following exact sequence of sheaves on $Y$
\begin{equation*}
    0 \to \pi_*\pi^*(L^{\oplus 2}) \to [\mathcal{O}_Y(2L) \oplus \mathcal{O}_Y(n'L)] \otimes \pi_*\mathcal{O}_X \to \pi_*(N_{\pi}) \to 0
\end{equation*}
\end{theorem}

\noindent\textit{Proof.} Let $V = L \oplus L$. By abuse of notation, we use $V$ to denote the total space of the vector bundle as well.   We have the following exact sequences (see \cite{CP})
\begin{equation*}
    0 \to \pi_*T_X \to \pi_*(T_V|_X) \to [\mathcal{O}_Y(2L) \oplus \mathcal{O}_Y(n'L)] \otimes \pi_*(\mathcal{O}_X) \to 0
\end{equation*}
\begin{equation*}
    0 \to \pi_*(T_X) \to \pi_*\pi^*(T_Y) \to \pi_*(N_{\pi}) \to 0
\end{equation*}
\begin{equation*}
    0 \to \pi_*(T_{V/Y}|_X) \to \pi_*(T_V|_X) \to \pi_*\pi^*T_Y \to 0
\end{equation*}

The exact sequences above induce a commutative diagram with exact rows and columns.
\[ 
\begin{tikzcd}
      & & 0 \arrow{d} & & \\
     & & \pi_*(T_{V/Y}|_X) \arrow{d} & & \\
     0 \arrow{r} & \pi_*T_X \arrow{r} \arrow{d} & \pi_*(T_V|_X) \arrow{d} \arrow{r} & (\mathcal{O}_Y(2L) \oplus \mathcal{O}_Y(n'L)) \otimes \pi_*(\mathcal{O}_X) \arrow{d} \arrow{r} & 0   \\
     0 \arrow{r} & \pi_*T_X  \arrow{r}  & \pi_*\pi^*(T_Y) \arrow{r} \arrow{d} & \pi_*(N_{\pi}) \arrow{r} & 0 \\
     & & 0 & & 
\end{tikzcd} 
\]
Applying the snake lemma and noting that $(T_{V/Y}|_X) = \pi^*(L^{\oplus 2})$, we have our exact sequence.\QEDB\par 

\begin{corollary}
In the situation of Set-up ~\ref{setup1} (1), let $\pi$ be a smooth, simple dihedral cover of $Y$. Then $H^1(\mathcal{N}_{\pi})=0$, i.e., (2.2) of Set-up~\ref{setup1} is satisfied.
\end{corollary}

\noindent\textit{Proof.} The corollary follows from the short exact sequence of Theorem ~\ref{sedn}, projection formula, ~\eqref{pushdn}, and Remark ~\ref{int} (4).\QEDB\par 

\color{black}

\subsection{Subcanonical covers}
If $\pi$ is simple cyclic (see Remark~\ref{remark.cyclic.subcanonical}), an iteration of simple cyclic covers as in Proposition~\ref{prop.ci} (i), or simple dihedral (see \eqref{pushdn}), the canonical bundle of $X$ is a pullback of a line bundle of $\mathbb P^N$. This property generalizes subcanonical subvarieties. We collect \textcolor{black}{here all these related definitions}:   

\begin{definition}\label{defsubcan}
Let $\hat X$ be a smooth projective variety and let $\hat \varphi:\hat X\to\mathbb{P}^N$ be a morphism. Let
\begin{equation*}
   L={\hat \varphi}^*\mathcal{O}_{\mathbb{P}^N}(1)\textrm{ and }s\in\mathbb{Z}.
\end{equation*}
\begin{enumerate}
    \item The polarized variety $(\hat X,L)$ is said to be {\it $s$--subcanonical} if $K_{\hat X}=\varphi^*\mathcal{O}_{\mathbb{P}^N}(s)$.
    \item  The morphism $\hat \varphi$ is called $s$--subcanonical if $(\hat X,L)$ is $s$--subcanonical and $\hat \varphi$ is induced by the complete linear series $|L|$.
    \item If $\hat X$ is embedded in $\mathbb P^N$, then $\hat X$ is an
    $s$--subcanonical subvariety if $K_{\hat X}=\mathcal{O}_{\mathbb{P}^N}(s)$. \end{enumerate}
\end{definition}

\begin{remark}\label{remark.subcanonical}
We note that, an s–subcanonical polarized scheme $(\hat X, L)$ is
\begin{itemize}
    \item[(1)] a Fano polarized scheme of index $-s$, in the sense of Fujita 
    (see \cite[Definition 1.5]{Fujita}), if $s < 0$;
    \item[(2)] a polarized scheme of Calabi–Yau, if $s = 0$;
    \item[(3)] a canonically polarized scheme of general type, if $s = 1$;
    \item[(4)] an s–subcanonical polarized scheme of general type, if $s > 0$.
\end{itemize}
\end{remark}

Proposition~\ref{nonexistence} and Theorems~\ref{birational}, \ref{theorem.embedding}
and \ref{2:1}, when applied to a subcanonical morphism $\varphi$, produce subcanonical morphisms, and, when applicable, subcanonical subvarieties: 

\begin{proposition}\label{prop.subcan.subvar}
In the situation of Set-up ~\ref{setup1}, 
\textcolor{black}{assume $\varphi$ is $s$-subcanonical.}
Let $\Phi:\mathcal{X}\to \mathbb{P}^N_T$ be a flat family of deformations of $\varphi:X\to\mathbb{P}^N$ over a smooth \textcolor{black}{variety} $(T,0)$. After shrinking $T$ if necessary, 
the morphism $\Phi_t$ is $s-$subcanonical. In particular, if 
 $\varphi': X' \longrightarrow \PP^N$ is a general element of the 
algebraic formally semiuniversal deformation space of $\varphi$, then $\varphi'$ is $s$-subcanonical. 
 In particular, if $\Phi_t$ is an embedding, then $\Phi_t(\mathcal X_t)$ is an $s$-subcanonical subvariety. 
\end{proposition}

\noindent\textit{Proof.}
The result follows from Remark~\ref{cls}, the regularity of $X$ (see Remark~\ref{int} (6)) and the fact that
$R^2p_{*}\mathbb Z_2$ is locally constant, where $p: \mathcal X \longrightarrow T$ is the deformation of $X$ induced by $\Phi$. \QEDB

\section{Necessary and sufficient conditions for
non--complete intersections
}\label{secnonci}

Theorem~\ref{theorem.embedding} gives a systematic way to construct
subvarieties of projective space, starting from complete intersection subvarieties. It is therefore natural to ask whether the subvarieties so obtained are complete intersections or not.
The answer is that both things may occur.  Next  Propositions~\ref{nci} and \ref{prop.ci} and Corollary~\ref{kpr}  give sufficient conditions for each situation to happen.
In either of the two cases, we show that, in the boundary of the irreducible Hilbert component of a subvariety produced by using  
Theorem~\ref{theorem.embedding},  lie points that correspond to ropes. Note that a rope of multiplicity bigger than $2$ is not a complete intersection, neither locally and, therefore, nor globally. 

\smallskip

The following proposition gives a sufficient condition for a 
subvariety obtained from a deformation of $\varphi$ to be  
a non--complete intersection.

\begin{proposition}\label{nci}
In the situation of Set-up ~\ref{setup1}, 
\textcolor{black}{assume $\varphi$ is $s$-subcanonical,}
Assume $\Phi:\mathcal{X}\to \mathbb{P}^N_T$ is a flat family of deformations of $\varphi:X\to\mathbb{P}^N$ over a smooth \textcolor{black}{variety} $(T,0)$, such that $\Phi_t:\mathcal{X}_t\to\mathbb{P}^N$ is an embedding for $t\neq 0$. \textcolor{black}{Suppose there are no integers $d'_1, d'_2, \dots, d'_r$, $d'_i > 1$ such that 
\begin{equation*}
   \textrm{(a) } \sum\limits_{i=1}^{r}d_i'= s + N +1 \textrm{, and (b) }\prod\limits_{i=1}^{r}d_i'=n\prod\limits_{i=1}^{r}d_i.
\end{equation*}
Then, shrinking $T$ if necessary,  $\Phi_t(\mathcal{X}_t)$ is not a complete intersection if $t \neq 0$.}
\end{proposition}
\noindent\textit{Proof.} 
\textcolor{black}{It follows from  Proposition~\ref{prop.subcan.subvar} that, after shrinking $T$ if necessary, $\Phi_t(\mathcal{X}_t)$ is $s-$subcanonical. Suppose now that $\Phi_t(\mathcal{X}_t)$, $t \neq 0$, is a complete intersection of multidegree $\underline{d}=(d'_1,d'_2,\cdots,d'_r)$.} Since, $K_X'=\Phi_t^*\mathcal{O}_{\mathbb{P}^N}(\sum\limits_{i=1}^r d_i'-N-1)$ , we get that $\sum\limits_{i=1}^r d_i'-N-1=\textcolor{black}{s}$, consequently (a) \textcolor{black}{ should hold}. To see \textcolor{black}{that (b) should also hold}, notice that \textcolor{black}{
$n\prod\limits_{i=1}^{r}d_i$ is the top self-intersection of the pullback, by $\varphi$ to $X$, of the hyperplane section of $\mathbb P^N$ and $\prod\limits_{i=1}^{r}d'_i$ is the top self-intersection of the pullback, by $\Phi_t$ to $\mathcal X_t$, of the hyperplane section of $\mathbb P^N$, so they are equal.} The proof is now complete.\QEDB\par 

\vspace{5pt}

We now show that codimension two examples produced by Theorem ~\ref{theorem.embedding}  (a) are always complete intersections. We \textcolor{black}{ prove} this fact by means of the following lemma and its corollary. In what follows, for a subvariety $j:\mathbb{Z}\hookrightarrow\mathbb{P}^{M}$, we will have $\mathcal{O}_Z(1)=j^*\mathcal{O}_{\mathbb{P}^M}(1)$, and $H^i_*(\mathcal{F}):=\bigoplus_{\nu\in\mathbb{Z}} H^i(\mathcal{F}(\nu))$.
\begin{lemma}\label{k}
In the situation of Set-up ~\ref{setup1}, assume the hypothesis of Theorem ~\ref{theorem.embedding} (a) or (b) is satisfied. 
Let $(T,0)$ be a smooth irreducible curve satisfying the conditions 
(I), (II), and (III) of Theorem ~\ref{theorem.embedding}. 
Then, possibly after shrinking $T$, $H^i_*(\mathcal{I}_{X_t})=0$ 
for all $2\leq i\leq m$ and $t\neq 0$, where $\mathcal{I}_{X_t}$ 
is the ideal sheaf of $X_t$ inside $\mathbb{P}^N_t$.
\end{lemma}
\noindent\textit{Proof.} We have the short exact sequence 
$0\to \mathcal{I}_{X_{t}}\to\mathcal{O}_{\mathbb{P}^N_{t}} 
\to \mathcal{O}_{X_{t}}\to 0$. 
It is easy to see that $H^i(\varphi^*\mathcal{O}_{\mathbb{P}^N}(k))=0$ 
for all $1\leq i\leq m-1$, $k\in\mathbb{Z}$. The conclusion follows by 
semicontinuity and the fact that 
$H^i_*(\mathcal{O}_{\mathbb{P}^N_{t}})=0$ 
for all $2\leq i\leq m$.\QEDB\par 

\begin{corollary}\label{kpr}
In the situation of Set-up ~\ref{setup1}, assume $r=2$ 
\textcolor{black}{and 
$m \geq 4$.}
Assume the hypothesis of Theorem ~\ref{theorem.embedding} (a) is satisfied, 
and let $(T,0)$ be a smooth irreducible curve satisfying the 
conditions (I), (II), and (III) of Theorem ~\ref{theorem.embedding}. 
Then (shrinking $T$ if necessary), 
$\Phi_{t}:X_{t}\hookrightarrow\mathbb{P}^N_{t}$ embeds 
$X_{t}$ as a complete intersection, 
for $t\neq 0$. \end{corollary}
\noindent\textit{Proof.} 
\textcolor{black}{It follows from the Barth-Larsen theorem (see 
\cite{Larsen}, see also
\cite[Corollary 3.2.3]{positivityI}) that any line bundle on $\mathcal X_t$ extends to $\PP^N$.}
Notice that 
codim$(\mathcal X_{t}/\mathbb{P}^N_{t})=2$. 
The following exact sequence follows from \cite[Theorem 5.1.1]{OSS} ;
\begin{equation*}
    0\to \mathcal{O}_{\mathbb{P}^N_{t}}\to 
    \mathcal{E}'_{t}\to \mathcal{I}_{\mathcal X_{t}}(l)\to 0,
\end{equation*}
where $\mathcal{E}'_{t}$ is the rank 2 bundle associated to 
$X_{t}$, and det$(\mathcal{N}_{\mathcal X_{t}/\mathbb{P}^N_{t}})=
\mathcal{O}_{\mathbb{P}^N_{t}}(l)\vert_{\mathcal X_t}$. 
We also know from \cite[Lemma 5.2.1]{OSS} that 
$\mathcal{E}'_{t}$ is split if and only if $X_{t}$ is a complete intersection. To this end, we apply \cite[Theorem 1]{KPR}. We obtain from Lemma ~\ref{k} that 
$H^i_*(\mathcal{I}_{X_{t}}(k))=0$ for $2\leq i\leq N-2$. Thus, 
$H^i_*(\mathcal{E}'_{t})=0$ for all $2\leq i\leq N-2$, since 
$H^i_*(\mathcal{O}_{\mathbb{P}^N_t})=0$ when $i$ is in the same range.\QEDB\par

\vspace{5pt}

 Proposition~\ref{prop.ci} below, which was inspired by conversations with Nori, shows that some of the ropes that appear in 
Theorem~\ref{theorem.embedding} (\textit{only} the ones satisfying (i) and (ii) below)
correspond to points that lie in  an irreducible component of the Hilbert scheme
whose general points correspond to smooth complete intersections. 
However, the general 
arguments below cannot determine 
if the general members of \textit{every} one--parameter smoothing of those ropes
are complete intersections (see Question~\ref{question.ci}). 
 
\begin{proposition}\label{prop.ci}
Let $X, Y, \varphi, \pi, n$ and $\mathcal E$ be as in Set-up~\ref{setup1}. 
Assume that 
\begin{enumerate}
 \item[(i)] $\pi=\pi_l \circ \cdots \circ \pi_1$ is the composition of simple 
cyclic covers $\pi_1, \dots, \pi_l$ such that, for each $1 \leq l' \leq l$, $\pi_{l'}$ is 
branched along  the pull back by $\pi_{l'} \circ \cdots \circ \pi_1$  of a divisor  
of $|\mathcal O_Y(n_{l'}\kappa_{l'})|$ ($n=n_1 \cdots n_l)$; 
\item[(ii)]
the  
unordered 
 multidegree of $Y$ is
$$\underline d_{\textrm{unord}}=(\kappa_1, \dots, \kappa_l, 
\beta_1, \dots, \beta_{r-l}).$$ 
\end{enumerate}
Furthermore, assume the hypotheses of Theorem~\ref{theorem.embedding} 
are satisfied. Then: 
 \begin{enumerate}
  \item A general member of 
  the algebraic formally semiuniversal deformation of $\varphi$ is 
  an embedding whose image is 
a complete intersection of unordered multidegree 
$$\underline d'_{\textrm{unord}}=(n_1\kappa_1, \dots, n_l\kappa_l,  \beta_1, 
\dots, \beta_{r-l}).$$ 
\item  If $\widetilde Y \hookrightarrow \mathbb P^N$ is an embedded rope  on $Y$ 
with conormal bundle $\mathcal E$ and $\Phi$ is a flat family of morphisms  satisfying 
(I), (II) and (III) of Theorem~\ref{birational}, then $\widetilde Y$ and, for any
$t \neq 0$, $\Phi_t(\mathcal X_t)$,  
correspond to points of an irreducible component  
    of the Hilbert scheme whose general point corresponds
    to a smooth, complete intersection subvariety 
of unordered multidegree $\underline d'$. 
 \end{enumerate}
\end{proposition}

\noindent\textit{Proof.}
We do in detail the proof when $\pi$ is simple cyclic, the general case follows from 
iterating the arguments used to prove the simple cyclic case. 

\smallskip
Thus, let $\pi$ be a simple cyclic cover branched along a (smooth) divisor  
of $|\mathcal O_Y(nk)|$, for some $k \in \mathbb Z$, $k > 0$. 
Recall that $Y$ is a complete intersection $H_1 \cap \cdots \cap H_r$ of multidegree
$$\underline d_{\textrm{unord}}=(k,  \beta_1, \dots, \beta_{r-1})$$ 
and let $k$ be the degree of $H_1$. 
 Let $Y'=H_1' \cap H_2 \cap \cdots \cap H_r$ be a 
 smooth complete intersection of  unordered 
 multidegree
 $$\underline d'_{\textrm{unord}}=(nk, k \beta_1, 
\dots, \beta_{r-1}),$$ 
where 
 $H_1'$ has degree $nk$. 
 By letting $H_1'$ degenerate to $n$ times $H_1$, we obtain 
 a smooth algebraic curve $S$, with a distinguished point $0 \in S$
and a flat family $\mathcal Y$
 of subschemes of $\PP^N$ 
over $S$, whose general member of $\mathcal Y$ is a smooth
 complete intersection of  unordered 
 multidegree $\underline d'$
and whose member at $0$ is a primitive multiple 
 structure $\widehat Y$ of multiplicity $n$, supported on $Y$.
 In fact, $H_1'$ can be chosen in such a way that, after base change, normalization and, if necessary,
 a linear automorphism of $Y$, 
 we obtain a 
 flat family 
 $\Psi$ of morphisms to $\PP^N$, over an algebraic curve $T'$ with 
 a distinguished point $0 \in T'$, such that $\Psi_0=\varphi$ 
 and $\Psi_t$ is 
 an embedding whose image is a smooth, complete intersection subvariety of 
 unordered multidegree $\underline d'$. 
 Then there is a point in the base $Z$ of 
 the algebraic formally semiuniversal deformation of $\varphi$
 (see Lemma~\ref{unobs}) 
 which corresponds to an embedding whose image is a smooth, 
 complete intersection variety.
Since being a complete intersection is an open property
 (see \cite{Ser75}), then there is a non--empty open set $U$ of 
 $Z$
 consisting of points that correspond to 
embeddings (see Theorem~\ref{theorem.embedding} (2)), whose images are smooth, 
complete intersection subvarieties. 
 This proves (1).

 \smallskip 
 Recall that $Z$ is irreducible (see Lemma~\ref{unobs}). 
 Hence there is a rational map $\rho$ from $Z$ to the Hilbert scheme. 
 Now $\textrm{Im}(\rho)$ is irreducible and is hence contained inside a 
 unique irreducible component $H$ of the Hilbert scheme. 
Any rope  $\widetilde{Y}$ as in the statement of (2) corresponds to a point 
in the closure of 
 $\textrm{Im}(\rho)$ and is hence contained in $H$. Also, for any $t \neq 0$, 
  any $\Phi_t(X_t)$ as in the statement of (2) 
  corresponds to a point in  $\textrm{Im}(\rho)$, which is 
therefore a point of $H$. By part (1), 
 $\textrm{Im}(\rho)$ contains at least one subvariety which is a complete intersection
 subvariety of multidegree $\underline d'$. 
 Since, by \cite{Ser75}, being a complete intersection is an open property 
 in the Hilbert scheme, a general point of $H$ corresponds to  
 a complete intersection
 subvariety of multidegree $\underline d'$. This proves (2).\QEDB

\begin{remark}
{\rm To prove 
Proposition~\ref{prop.ci} (1)  we use, among other things, 
an elementary construction from which we obtain the flat family 
$\Psi$. This construction is a particular case of the one in 
\cite{CL} (see \cite[Definition 2.3, Theorem 2.4 and Remark 2.5]{CL}).} 
M. Nori also showed us a similar example.
\end{remark}

\begin{question}
{\rm Let $\mathcal E$ be as in Proposition~\ref{prop.ci}. If $n > 2$, 
there is one irreducible component of the Hilbert containing points corresponding to two
different kind of multiple structures, namely,  ropes embedded in $\PP^N$, 
supported on $Y$ with conormal bundle $\mathcal E$ on the one hand, 
and complete intersection multiple structures obtained by intersecting 
multiple hypersurfaces
and smooth hypersurfaces on the other hand. In the case in which 
$\pi$ is simple cyclic, then the latter multiple structures are  
 primitive (and are like $\widehat Y$). Thus, a natural question to ask is how, 
 in the Hilbert scheme, the loci parameterizing each kind of multiple structure are related.}
\end{question}

\color{black}

\section{Deforming finite morphisms to construct small codimension subvarieties}\label{4}

\textcolor{black}{In this section we use the results proven in the previous sections to produce, in a systematic way,   {sub}varieties of projective space, {of} infinitely many degrees, for any given codimension. In particular, we produce infinitely many   {sub}varieties {of} small codimension {in $\mathbb{P}^N$}. In order to do so, we deform morphisms to $\mathbb{P}^N$, finite onto their image, to embeddings. More precisely, in this section we study the deformations of simple cyclic covers of complete intersections (in Sections~\ref{secz} and \ref{5} we extend this study to more general abelian covers and to dihedral covers).} We first describe our set-up, \textcolor{black}{that we will follow in this and in Sections~\ref{section.cyclic.covers.nci} and \ref{section.cyclic.covers.non.embeddings}.}

\begin{set-up}\label{subc}
In the situation of Set-up ~\ref{setup1}, \textcolor{black}{let $k \in  \mathbb{Z}_{>0}$ and} assume $\pi$ is simple 
cyclic\textcolor{black}{, branched along a smooth divisor in $|\mathcal O_Y(nk)|$.}
Consequently, 
\color{black}
\begin{equation*}
    \mathcal E= \mathcal O_Y(-k) \oplus \cdots \oplus \mathcal O_Y(-(n-1)k), 
\end{equation*}
and, therefore,  $k_i=ik$
\color{black}
for
all $i=1, \dots, n-1$. 
\end{set-up}
We now make \textcolor{black}{two} remarks. The first one \textcolor{black}{recalls that simple cyclic covers are subcanonical and tells what kind of variety $X$ is, depending on $Y$ and the cover}. The \textcolor{black}{second} one shows \textcolor{black}{in particular} that, \color{black} when $\varphi$ is the canonical morphism, the geometric genus of $X$ is bounded in terms of the dimension of $X$ and the degree of $\pi$. 

\begin{remark}\label{remark.cyclic.subcanonical}
Assume Set-up~\ref{subc} and let $L=\pi^*\mathcal{O}_Y(1)$.
Then $(X,L)$ is subcanonical. Indeed, by adjunction and 
the ramification formula (see e.g. \cite[Lemma I.17.1]{BHPV}), 
$K_X=\pi^*\mathcal{O}_Y(\delta+(n-1)k-N-1)$. In particular, 
 $(X,L)$ is  $s$--subcanonical if and only if 
 \begin{equation}\label{eq.subcan.simple.cyclic}
    N+1+s=\delta+(n-1)k. 
 \end{equation}
 \color{black}
 \textcolor{black}{From Remark~\ref{remark.subcanonical} we get the following facts:}
\begin{itemize}
    \item[$(1)$] $X$ is Fano variety if and only if $\delta+(n-1)k\leq N$ (then, $Y$ is also Fano). 
    In this case $(X,L)$ is a Fano polarized variety of index $-s$
   if and only if $N+1+s=\delta+(n-1)k$.
    \item[$(2)$] $X$ is a Calabi--Yau variety if and only if $\delta+(n-1)k=N+1$ 
    (in this case, $Y$ is Fano).
    \item[$(3)$] $X$ is a variety of general type if and only if $\delta+(n-1)k\geq N+2$. 
    The morphism $\varphi$ (respectively $(X,L)$)  is canonical \textcolor{black}{(respectively, the canonical polarization)} if and only if $\delta+(n-1)k=N+2$ 
        and $k\geq 2$ (resp. $\delta+(n-1)k=N+2$); in this case, $Y$ is Fano 
        (resp. $Y$ is Fano, unless $n=2$ and $k=1$, {in which case $Y$ is Calabi--Yau}).
\end{itemize}
\end{remark}

\begin{remark}\label{Nm}
In the situation of Set-up ~\ref{subc}, the following happens.
\begin{itemize}
    \item[$(1)$] If $(X,L)$ is $s$--subcanonical, then $m+1\leq N
    \leq 2m+s+1-(n-1)k \leq 2m+s-n+2$.
    \item[$(2)$] If $\varphi$ is $s$--subcanonical, then $m+1\leq N\leq 2m+s-2n+3$. \textcolor{black}{In particular, if $\varphi$ is a canonical morphism, then $p_g(X) \leq 2m-2n+4$.}
\end{itemize}
\end{remark}

\color{black}

\begin{remark}
Although in this sections we will use Theorem~\ref{theorem.embedding} (a) only to produce $s$-subcanonical morphisms, we note it can
\textcolor{black}{also} be used to produce morphisms induced by non--complete linear series, in the way shown in Examples~\ref{example.non.complete.series.embedding} and 
\ref{example.non.complete.series.birational}.
\end{remark}

\color{black}

\textcolor{black}{Thus, from now on in this section} we will only look at the cases in which 
$\varphi$ is $s$--subcanonical, in particular, 
it is induced by a complete linear series. 
Recall that in these cases $k_1\geq 2$, \textcolor{black}{i.e., $k \geq 2$}
(see Set-ups ~\ref{setup1} \textcolor{black}{and \ref{subc}} for notations). 
In this subsection we apply Theorem ~\ref{theorem.embedding} (a)
when $\pi$ is simple cyclic. As we will see, Theorem ~\ref{theorem.embedding} (a) produces in a systematic way, for any given dimension $m$, $m \geq 3$ and any given codimension $r$, subvarieties of infinitely many different degrees. In particular, it produces subvarieties of small given codimension $r$ and infinitely many different degrees.

\smallskip

First we study what invariants $m, n$ and $N$ in Set-up~\ref{subc} and $s$ in Definition~\ref{defsubcan} satisfy the hypothesis of   Theorem ~\ref{theorem.embedding} (a).

\begin{proposition}\label{lemmaemb}
\color{black}
In the situation of Set-up ~\ref{subc}, \color{black}
if 
the hypothesis of Theorem ~\ref{theorem.embedding} (a) 
holds and 
$\varphi$ is $s$--subcanonical, then,
\color{black}
\begin{equation}\label{eq.lemmaemb}
 m+n-1\leq N\leq 2(m+n\textcolor{black}{-1})+s-(n-1)(n+2) \textcolor{black}{+1}  \textcolor{black}{=2m -n(n-1) + s +1},
\end{equation}
so, in particular, 
\begin{equation}
    s\geq \textcolor{black}{n^2-m-2}.
\end{equation}

\color{black}
\smallskip
In particular, 
\begin{equation}\label{eq.lemmaemb.2}
   m+1 \leq N \leq 2m+s-1, 
\end{equation}
and, if $n \geq 3$, then 
\begin{equation}\label{eq.lemmaemb.3}
   m+2 \leq N \leq 2m+s-5. 
\end{equation}
Further, if $s=(n-1)(n+2)-(m+n)$, then $N=m+n-1$ and $\underline{d}=(2,4,\cdots\cdots,2(n-1))$.
\end{proposition}

\color{black}
\noindent\textit{Proof.} Assume 
the hypothesis of Theorem ~\ref{theorem.embedding} 
 (a) holds. Then, 
obviously $r=N-m\geq n-1$, which gives the lower bound of $N$ in 
\eqref{eq.lemmaemb}.
Under the assumption, the \textit{unordered} 
multidegree of $Y$ is the following; 
\begin{equation}\label{eq.lemmaemb.pr1}
\underline{d}_{\textrm{unord}}=
(k,2k,\cdots,(n-1)k,\beta_1,\beta_2,\cdots,
\beta_{N-m-n+1}),    
\end{equation}
 where each $\beta_i\geq 2$. 
Thus, $\delta\geq \frac{k(n-1)n}{2}
+2(N-m-n+1)$. 
Since $\delta+k(n-1)=N+s+1$, we get the following;
\begin{equation}\label{eq.lemmaemb.pr2}
    N+s+1\geq\frac{k(n-1)n}{2}+2(N-m-n+1)+
k(n-1).
\end{equation}
An elementary computation completes 
the proof of \textcolor{black}{\eqref{eq.lemmaemb},} since $k\geq 2$.
\textcolor{black}{If we set $n=2$ in \eqref{eq.lemmaemb} (respectively $n=3$), then \eqref{eq.lemmaemb} becomes \eqref{eq.lemmaemb.2} (respectively \eqref{eq.lemmaemb.3}). Consequently,   \eqref{eq.lemmaemb.2} and \eqref{eq.lemmaemb.3} are obvious consequences of \eqref{eq.lemmaemb}.} The \textcolor{black}{last} assertion {is} a consequence of \textcolor{black}{\eqref{eq.lemmaemb}, \eqref{eq.lemmaemb.pr1} and \eqref{eq.lemmaemb.pr2}}.\QEDB\par

\color{black} 

\smallskip

Now we show that, for all set of invariants satisfying 
\eqref{eq.lemmaemb}, there exist morphisms $\varphi$ to which Theorem~\ref{theorem.embedding} (a) can be applied, therefore 
producing infinitely many subvarieties in the range 
$3 \leq m  \leq N-1$:

\begin{theorem}\label{theorem.embedding.a}
Given any integers
$n, m, s$ and $N$ such that  $m  \geq 3$ 
and $n \geq 2$,  if  \eqref{eq.lemmaemb} holds,
then there 
exist  
smooth varieties $X'$ of dimension $m$ and 
$s$-subcanonical embeddings $\varphi':X' \longrightarrow \mathbb P^N$  such that 
\begin{enumerate}
    \item[(a)] the morphisms $\varphi'$ are 
    deformations of  morphisms $\varphi$, where 
$\varphi, m$ and $N$ are as in Set-up~\ref{subc} and $\varphi$ satisfies the hypothesis of Theorem~\ref{theorem.embedding} (a); 
\item[(b)] the subvarieties $\varphi'(X')$ are one-parameter deformations, as described in Theorem~\ref{theorem.embedding}, of  multiplicity $n$ rope subschemes.
\end{enumerate}
For any given integers
$n, m, N$ and $s$ satisfying   $m  \geq 3$, 
$n \geq 2$ and \eqref{eq.lemmaemb}, there are infinitely many non--isomorphic subvarieties $\varphi'(X')$ as above. 
\end{theorem}

\color{black}
\noindent\textit{Proof.} Let $m,n, s \in\mathbb{Z}$, with $m\geq 3$, $n\geq 2$ satisfying 
\eqref{eq.lemmaemb}.
\color{black}
For any integer $k$ such that 
\begin{equation*}
    2 \leq k \leq \frac{2(2(m+n)+s-N-1)}{(n-1)(n+2)},
\end{equation*}
(e.g., for $k=2$), 
\color{black}
then there are integers $\beta_1,\beta_2,\cdots,
\beta_{N-m-n+1}\geq 2$ satisfying the equation
$$\sum\beta_i+\frac{\textcolor{black}{k}n(n-1)}{\textcolor{black}{2}}+
\textcolor{black}{k}(n-1)=N+s+1.$$ For any such choices 
of $\beta_i$'s, let $Y$ be a complete intersection 
in $\mathbb{P}^N$ of multidegree 
$$\textcolor{black}{\underline{d}_{\textrm{unord}}=(k,2k,\cdots\cdots,(n-1)k,\beta_1,\cdots,\beta_{N-m-n+1}).}$$
Let $\pi:X\to Y$ 
be a simple cyclic cover branched along a 
smooth member of $|\mathcal{O}_Y(\textcolor{black}{k}n)|$. 
Then the corresponding morphism
$\varphi$ is $s$--subcanonical 
and satisfies the hypothesis of 
Theorem ~\ref{theorem.embedding} (a). Thus, 
a general deformation of $\varphi$ is an 
{embedding}.\QEDB

\begin{remark}\label{remark.moduli}
\begin{enumerate}
    \item For $m, N, n$ fixed with $m \geq 3, n \geq 2$, $N \geq m + n-1$, there exist smooth varieties $X'$ of dimension $m$ as in Theorem~\ref{theorem.embedding.a} belonging to infinitely many
   different moduli spaces. This is because we can choose the $d_i$'s and/or $k$  arbitrarily large (then $s$ also grows), so $K_X^m$ also grows arbitrarily. These moduli spaces possess reduced and irreducible components with a locally closed locus that parametrizes smooth varieties with an  $s$-subcanonical morphism, of degree $n$ and finite onto its image, whereas the general points of the components correspond to smooth varieties with an $s$-subcanonical morphism which is an embedding. In Section~\ref{section.moduli} we will describe more specifically this phenomenon in the case of the canonical map, i.e., if $s=1$ (see Corollary~\ref{def of can morphisms}).  
   
   \item For $m, n$ fixed with $m \geq 3, n \geq 2$ or, simply, for $m$ fixed with $m \geq 3$, if we let $s$  grow arbitrarily, then $N$, and, hence, $p_g$, can be chosen arbitrarily large. Therefore there exist smooth varieties $X'$ of dimension $m$ as in Theorem~\ref{theorem.embedding.a} with $p_g(X')$ arbitrarily large.
\end{enumerate}
\end{remark}

\vspace{5pt} 

Theorem~\ref{theorem.embedding.a} yields this corollary:

\begin{corollary}\label{cor.embedding.a}
Given any integers $m$ and $N$ such that $3 \leq m \leq N-1$, 
there 
exist 
smooth varieties $X'$ of dimension $m$ and embeddings $\varphi':X' \longrightarrow \mathbb P^N$, with  $\varphi'(X')$  having infinitely many different degrees, such that 
\begin{enumerate}
    \item[(a)] the morphisms $\varphi'$ are
    deformations of  morphisms $\varphi$, where 
$\varphi, m$ and $N$ are as in Set-up~\ref{subc} and $\varphi$ satisfies the hypothesis of Theorem~\ref{theorem.embedding} (a); 
\item[(b)] the subvarieties $\varphi'(X')$ are one-parameter deformations, as described in Theorem~\ref{theorem.embedding}, of  rope subschemes.
\end{enumerate}
 More precisely:
 \begin{enumerate}
 \item If 
\begin{equation}\label{eq.embedding.4}
     s \geq N-2m +1,
 \end{equation} then 
 there exist $s$-subcanonical embeddings $\varphi'$ as above and, 
for any such $s$, there are infinitely many non--isomorphic subvarieties $\varphi'(X')$ as above. 
     \item If 
     \begin{equation}\label{eq.embedding.5}
    m \leq N-2 \ \ \mathrm{ and} \ \  s \geq N-2m +5,
      \end{equation}
      then the above ropes can be chosen to be of multiplicity greater than 2, and hence not a complete intersection.
 \end{enumerate}
\end{corollary}

\noindent\textit{Proof.}
Under our assumption that $m \leq N-1$,   \eqref{eq.lemmaemb.2} is equivalent to \eqref{eq.embedding.4}.
As we saw in the proof of Proposition~\ref{lemmaemb}, if we set $n=2$ in \eqref{eq.lemmaemb}, then we obtain \eqref{eq.lemmaemb.2}, so setting 
$n=2$ in Theorem~\ref{theorem.embedding.a}, we obtain morphisms $\varphi'$ as required in (1). Furthermore, for any $s$ satisfying \eqref{eq.embedding.4}, let $\delta$ satisfy
\eqref{eq.subcan.simple.cyclic}. Then, as $s$ goes to infinity, we have a sequence of values of $\delta$ that gives rise to a sequence of $s$-subcanonical morphisms $\varphi$ as in  Theorem~\ref{theorem.embedding.a} such that $L^m$ goes to infinity, for 
$L=\varphi^*\mathcal O_{\mathbb P^N}(1)$. Finally, \eqref{eq.lemmaemb.3} is equivalent to \eqref{eq.embedding.5}. If we set  $n=3$ in \eqref{eq.lemmaemb}, then we obtain \eqref{eq.lemmaemb.3}, so setting 
$n=3$ in Theorem~\ref{theorem.embedding.a}, we obtain morphisms $\varphi'$ as required in (2), since ropes of multiplicity $n \geq 3$ are locally not a complete intersection.\QEDB

\vspace{5pt}

\textcolor{black}{To give a taste of the power of Theorem~\ref{theorem.embedding.a} and Corollary~\ref{cor.embedding.a}, we construct} some explicit examples of small codimension subvarieties of Corollary~\ref{cor.embedding.a} (2). We will  consider only examples in which $n \geq 3$ so all the subvarieties in them are one-parameter deformations of locally non--complete intersections and, therefore, non--complete intersections, embedded ropes.

\begin{example}
For fixed $m \geq 3$ and $n \geq 3$, we look at the subvarieties  $\varphi'(X')$ with smallest possible $s$ (that is, according to 
Proposition~\ref{lemmaemb}, 
 $s=\textcolor{black}{n^2-m-2}$), constructed in Theorem~\ref{theorem.embedding.a}. By Proposition~\ref{lemmaemb}, these subvarieties  have codimension $r=n-1$ and $k=2$.
Due to Corollary ~\ref{kpr}, \textcolor{black}{the subvarieties $\varphi'(X')$ corresponding to} the rows shaded with light blue color in the table below are complete intersections. For the white rows, the general deformation 
of $\varphi$ is an embedding whose image is a complete intersection, due to 
Proposition~\ref{prop.ci}. \textcolor{black}{A priori} it is not known if the same is true for special deformations (see Question~\ref{question.ci}), \textcolor{black}{although, they should also be complete intersections according to Hartshorne's conjecture. Further, these  are the lowest codimension examples of subvarieties which are smoothings of non--complete intersections ropes.} 
The number of different invariants of such varieties $\varphi'(X')$
is infinite, so we will just list a few of them in the table below, 
precisely those of codimension $2$ and $3$ and $-2
\leq s \leq 2$. 

\vspace{5pt}

\begin{center}
 \begin{tabular}{c|c|c|c|c|c|c|c|c} 
 \hline
 $m$ & $n$ & $k$ & $N$ & $s$ & $\underline{d}$ & \textcolor{black}{$\mathrm{deg}(\varphi'(X'))$} & \textcolor{black}{$K_{X'}^m$} 
 & \textcolor{black}{$p_g(X')$} \\ 
 \hline\hline
 \rowcolor{LightCyan}
 $9$ & $3$ & $2$ & $11$ & $-2$ & $(2,4)$ & \textcolor{black}{$24$}  & $-24\cdot 2^9$ 
 & $0$\\
 \hline
 $16$ & $4$ & $2$ & $19$ & $-2$ & $(2,4,6)$ &\textcolor{black}{$192$} 
 & $192\cdot 2^{16}$ & $0$\\
 \hline
 \rowcolor{LightCyan}
 $8$ & $3$ & $2$ & $10$ & $-1$ & $(2,4)$&\textcolor{black}{$24$} & $24$ & $0$\\
 \hline
  $15$ & $4$ & $2$ & $18$ & $-1$ & $(2,4,6)$&\textcolor{black}{$192$} & $-192$ & $0$\\
 \hline
 \rowcolor{LightCyan}
 $7$ & $3$ & $2$ & $9$ & $0$ & $(2,4)$&\textcolor{black}{$24$} & $0$ & $1$\\
 \hline
 $14$ & $4$ & $2$ & $17$ & $0$ & $(2,4,6)$&\textcolor{black}{$192$} & $0$ & $1$\\
 \hline
 \rowcolor{LightCyan}
 $6$ & $3$ & $2$ & $8$ & $1$ & $(2,4)$&\textcolor{black}{$24$} & $24$ & $9$\\
 \hline
 $13$ & $4$ & $2$ & $16$ & $1$ & $(2,4,6)$&\textcolor{black}{$192$} & $192$ & $17$\\
 \hline
 \rowcolor{LightCyan}
  $5$ & $3$ & $2$ & $7$ & $2$ & $(2,4)$&\textcolor{black}{$24$} & $24\cdot 2^5$ &  \textcolor{black}{$37$} \\ 
   \hline
 \textcolor{black}{$12$} & \textcolor{black}{$4$} & \textcolor{black}{$2$} & \textcolor{black}{$15$} & \textcolor{black}{$2$} & \textcolor{black}{$(2,4,6)$}&\textcolor{black}{$192$} & \textcolor{black}{$192 \cdot 2^{12}$} & \textcolor{black}{$137$}\\
 \hline
 \end{tabular}
 \end{center}

 \end{example}

 \color{black}
 \begin{example}
Theorem~\ref{theorem.embedding.a} allows also the construction of smooth subvarieties near the boundary of, but inside, the range of Hartshorne's conjecture, like the very small sample of the smooth subvarieties   in the range $r=(1/3)N-1$, displayed in Table \ref{t01} of the introduction, which are  obtained by deforming $\varphi$, which factors through a simple cyclic cover $\pi$ of degree $n$ of  $Y$, when $\pi$ is  branched along a smooth divisor in $|\mathcal O_Y(2n)|$. By Proposition~\ref{prop.ci}, the image of the general deformation of  $\varphi$  is a complete intersection, although we do not know whether the same is true for  some special deformations of $\varphi$ (see Question~\ref{question.ci}).
 \end{example}

\section{\textcolor{black}{Deforming finite morphisms to construct smooth non--complete intersection subvarieties}}\label{section.cyclic.covers.nci}

\color{black}
\textcolor{black}{In this section we construct smooth, non--complete intersection subvarieties of projective space, by applying Theorem ~\ref{theorem.embedding} (b)
when $\pi$ is a simple cyclic cover.} As we will see, Theorem ~\ref{theorem.embedding} (b) produces in a systematic way, for any given dimension $m$, $m \geq 3$ and any given codimension $r \geq m+1$, smooth subvarieties of infinitely many different degrees. 
Most importantly, for any $m$ and $r$ such that $r \geq m+1$, Theorem ~\ref{theorem.embedding} (b) produces smooth, non--complete intersection subvarieties, embedded by complete linear series. 
First we study what invariants $m, n$ and $N$ in Set-up~\ref{subc} and $s$ in Definition~\ref{defsubcan} satisfy the hypothesis of   Theorem ~\ref{theorem.embedding} (b).  

\smallskip

\color{black}

Although, as already remarked, our main interest are morphisms induced by complete linear series, we give first an example of how Theorem~\ref{theorem.embedding} (b) yields embeddings induced by non--complete linear series:

\begin{example}\label{example.non.complete.series.embedding}
Set $k=1$, and assume $N\leq \frac{nm+s}{n-1}$. Let $Y$ be a complete intersection of a 
hypersurface of degree $\alpha=nm+N(1-n)+2+s$ and $N-m-1$ hypersurfaces of degree $n$. 
By assumption, $\alpha\geq 2$. If one wants to write the multidegree of $Y$ that is consistent with the convention of Set-up ~\ref{setup1}, it will be
    \[ \underline{d} =
  \begin{cases}
    (\overbrace{n,\cdots\cdots,n}^\text{$N-m-1$}, \alpha); 
    & \textrm{ if }N\leq \frac{nm+2+s-n}{n-1},   \\
    (\alpha,\underbrace{n,\cdots\cdots,n}_\text{$N-m-1$}); & \textrm{ if }  N> \frac{nm+2+s-n}{n-1}. 
  \end{cases}\]

Recall that  $\pi: X \to Y$ is  a simple cyclic cover of $Y$ 
branched along a smooth divisor in $|\mathcal{O}_Y(n)|$ and $\varphi$ is  the corresponding morphism 
(which is induced by an incomplete linear series). \par
\smallskip

\noindent   If one of the following conditions hold;
    \begin{equation}
        2m+n-1\leq N\leq \textrm{min}
        \bigg\{\frac{nm+2+s}{n-1}-1,2m+s+2-n\bigg\} 
        \textrm{ or }
    \end{equation}
    \begin{equation}
        2m+n\leq N\leq 2m+s+2-n,
    \end{equation}
    then, since Theorem ~\ref{theorem.embedding}  (b) 
    applies, a general deformation of $\varphi$ is an 
    embedding.
\end{example}

From now on, in this section we look only at morphisms induced by complete linear series. 

\color{black}

\begin{proposition}\label{prop.emb.b}
\color{black}
In the situation of Set-up ~\ref{subc}, \color{black}
if 
the hypothesis of Theorem~\ref{theorem.embedding} (b) 
holds and 
$\varphi$ is $s$--subcanonical, then,
\color{black}
\begin{equation}\label{eq.lemmaemb.b}
 2m+n-1\leq N\leq 2(2m+n-1)-2(n-1)(m+n)+s+1,
\end{equation}
\textcolor{black}{so, in particular, 
\begin{equation}
s\geq 2m(n-2)+n(2n-3).    
\end{equation}
In particular, 
\begin{equation}\label{eq.lemmaemb.6}
   2m+1 \leq N \leq 2m+s-1. 
\end{equation}
and, if $n \geq 3$, then 
\begin{equation}\label{eq.lemmaemb.6b}
   2m+2 \leq N \leq s-7. 
\end{equation}}
Further, if $s=2m(n-2)+n(2n-3)$, then $N=2m+n-1$, $k=2$, and $Y$ has multidegree $$(\underbrace{2(n-1),\cdots\cdots,2(n-1)}_\text{$m+n-1$}).$$
\end{proposition}

\color{black}

\noindent\textit{Proof.} We just prove \eqref{eq.lemmaemb.b}, the other assertions will easily follow from it. Assume the hypothesis of Theorem ~\ref{theorem.embedding} (b) holds. Then, $2r\geq N+n-1$ and consequently $N\geq 2m+n-1$ since $r=N-m$. Now, under the assumption, $\delta$ will be least if the \textit{unordered} multidegree is the following;
$$\underline{d}_{\textrm{unord}}=(\underbrace{k(n-1),
\cdots\cdots,k(n-1)}_\text{$m+n-1$},\underbrace{2,\cdots\cdots,2}_\text{$N-2m-n+1$}\hskip -.1cm),$$
consequently, $\delta\geq k(n-1)(m+n-1)+2(N-2m-n+1)$. 
Since $\delta+k(n-1)=N+s+1$, we get,
$$N+s+1\geq k(n-1)(m+n)+2(N-2m-n+1).$$
Since $k\geq 2$, \eqref{eq.lemmaemb.b} follows.\QEDB\par

\vspace{5pt}

\color{black}
Now we see that, for all set of invariants satisfying 
\eqref{eq.lemmaemb.b}, there exist morphisms $\varphi$ to which Theorem~\ref{theorem.embedding} (b) can be applied, therefore 
producing infinitely many subvarieties in the range 
$3 \leq m  < \frac{N}{2}$.

\begin{theorem}\label{theorem.embedding.b}
Given any integers
$n, m, s$ and $N$ such that  $m  \geq 3$ 
and $n \geq 2$,  if \eqref{eq.lemmaemb.b} holds,
then there 
exist 
smooth varieties $X'$ of dimension $m$ and 
$s$-subcanonical embeddings $\varphi':X' \longrightarrow \mathbb P^N$  such that 
\begin{enumerate}
    \item[(a)] the morphisms $\varphi'$ are  
    deformations of  morphisms $\varphi$, where 
$\varphi, m$ and $N$ are as in Set-up~\ref{subc} and $\varphi$ satisfies the hypothesis of Theorem~\ref{theorem.embedding} (b); 
\item[(b)] the subvarieties $\varphi'(X')$ are one-parameter deformations, as described in Theorem~\ref{theorem.embedding}, of  multiplicity $n$ rope subschemes.
\end{enumerate}
For any given integers
$n, m, N$ and $s$ satisfying   $m  \geq 3$, 
$n \geq 2$ and \eqref{eq.lemmaemb.b}, there are infinitely many non--isomorphic subvarieties $\varphi'(X')$ as above. 
\end{theorem}

\noindent\textit{Proof.}
Let $m, n, s$ be integers satisfying $m\geq 3$, $n\geq 2$ and \eqref{eq.lemmaemb.b}. For any integer $k$ such that 
 \begin{equation*}
2 \leq k  \leq  \frac{2(2m+n)+s-N-1}{(n-1)(m+n)}, 
 \end{equation*}
 (e.g., $k=2$), 
 \color{black}
there exist integers $\beta_1,\beta_2,
\cdots,\beta_{N-2m-n+1}\geq 2$, and $\alpha_1,\cdots,\alpha_{m+n-1}\geq 
\textcolor{black}{k}(n-1)$ satisfying the  equation
$$\sum\beta_i+\sum\alpha_j+\textcolor{black}{k}(n-1)
=N+s+1.$$ For any such choices of $\alpha_i$'s and $\beta_j$'s, let $Y$ be a complete intersection in $\mathbb{P}^N$ of multidegree $$\underline{d}_{\textrm{unord}}=(\alpha_1,\cdots\cdots,\alpha_{m+n-1},\beta_1,\cdots,\beta_{N-2m-n+1}).$$
Let $\pi:X\to Y$ be a simple cyclic cover branched 
along a smooth member of $|\mathcal{O}_Y(\textcolor{black}{k}n)|$. 
Then $\pi$ is an $s$--subcanonical cover of 
degree $n$ satisfying the hypothesis of 
Theorem ~\ref{theorem.embedding}  (b). 
Thus, a general deformation of $\varphi$ is an {embedding}.\QEDB

\vspace{5pt}

We omit the proof of the following which is analogous to the proof of Corollary~\ref{cor.embedding.a}.

\begin{corollary}\label{cor.embedding.b}
Given any integers $m$ and $N$ such that 
$3 \leq m  < \frac{N}{2}$
there 
exist 
smooth varieties $X'$ of dimension $m$ and embeddings $\varphi':X' \longrightarrow \mathbb P^N$, with  $\varphi'(X')$  having infinitely many different degrees, such that 
\begin{enumerate}
    \item[(a)] the morphisms $\varphi'$ are 
    deformations of  morphisms $\varphi$, where 
$\varphi, m$ and $N$ are as in Set-up~\ref{subc} and $\varphi$ satisfies the hypothesis of Theorem~\ref{theorem.embedding} (b); 
\item[(b)] the subvarieties $\varphi'(X')$ are one-parameter deformations, as described in Theorem~\ref{theorem.embedding}, of  rope subschemes.
\end{enumerate}
 More precisely:
 \begin{enumerate}
 \item If 
\begin{equation}\label{eq.embedding.7}
     s \geq N-2m +1,
 \end{equation} then 
 there exist $s$-subcanonical embeddings $\varphi'$ as above and, 
for any such $s$, there are infinitely many non--isomorphic subvarieties $\varphi'(X')$ as above. 
     \item If 
     \begin{equation}\label{eq.embedding.8}
    m \leq \frac{N}{2}-1 \ \ \mathrm{ and} \ \  s \geq N+7,
      \end{equation}
      then the above rope subschemes can be chosen not to be   complete intersections.
 \end{enumerate}
\end{corollary}

\color{black} Next thing we want to know is which ones among the 
subvarities $\varphi'(X')$ of Theorem~\ref{theorem.embedding.b} are non--
complete intersections. \textcolor{black}{As we will see in
Theorem \ref{real.main0.5}, for each codimension in the range of Theorem~\ref{theorem.embedding.b}, there will be infinitely many of them. Now,    in the next example} we {start} displaying invariants of the lowest degree, non--complete intersection threefolds obtained by deforming double, triple and quadruple simple cyclic covers. 

\begin{example}\label{emb.more}
There exist $s$-subcanonical, non--complete intersection, smooth threefolds $\varphi'(X')$ in $\mathbb P^N$, embedded by complete linear series, where $\varphi'$ is a deformation of $\varphi$, and $\varphi$, $n$ and $k$ are as in Set-up~\ref{subc}, with the following invariants. 
\medskip

\begin{center}
\color{black}
 \begin{tabular}{c|c|c|c|c|c} 
 \hline
  $n$ & $k$ & $N$ & $s$ & $\underline{d}$ & 
 \color{black} deg$(\varphi'(X'))$\\
 \hline\hline
   \color{black}  \color{black} $2$ & \color{black} $2$ & \color{black} $7$ &\color{black}  $6$ & \color{black} $(3,3,3,3)$ & \color{black} $162$\\
 \hline
  \color{black}  \color{black} $2$ & \color{black} $2$ & \color{black} $7$ & \color{black} $7$ & \color{black} $(3,3,3,4)$ & \color{black} $216$\\
 \hline
 \color{black}
   $3$ & $2$ & $8$ & $15$ & $(4,4,4,4,4)$ & $3072$\\
 \hline
   $3$ & $2$ & $8$ & $16$ & $(4,4,4,4,5)$ & $3840$\\
 \hline
   $4$ & $2$ & $9$ & $32$ & $(6,6,6,6,6,6)$ & $186624$\\
 \hline
  $4$ & $2$ & $9$ & $33$ & $(6,6,6,6,6,7)$ & $217728$\\
 \hline 
 \end{tabular}
 \end{center}
 \medskip
 
\noindent  The subvarieties $\varphi'(X')$ are not complete intersections because of Proposition ~\ref{nci}. 
\end{example}

\par
\medskip

\color{black} 
In the next two theorems,  for any pair $(m,N)$ such that $3 \leq m < {N}/{2}$, 
we construct non--complete intersection, $m$--dimensional smooth subvarieties 
of $\mathbb P^N$ of infinitely many different degrees by 
deforming suitable simple cyclic covers of suitable complete 
intersection subvarieties. 
\color{black}
In particular, 
both theorems provide  subvarieties of dimension $m$ in $\mathbb P^N$ such that
the ratio $m/N$ goes to ${1}/{2}$ from below as $m$ approaches infinity.
The subvarieties constructed in Theorem~\ref{real.main0.5} and the subvarieties constructed in Theorem~\ref{main0.5} are different since, 
for instance, their degrees are different.

\color{black}

\begin{theorem}\label{real.main0.5}
Let $m$ and $N$ be any integers such that $3 \leq m  < \frac{N}{2}$ and let $p$ be any odd prime integer. 
\begin{enumerate}
    \item  There exist smooth, non--complete intersection subvarieties $X'_m$, of dimension $m$ and degree $2p^{N-m}$, embedded in $\mathbb P^N$ by a complete linear series.
   \item For any  $m, N, p$ as above, 
 there are subvarieties $X'$ as in (1) which are 
 $s$--subcanonical, where $s=(p-1)N-pm+k-1$, for any
 integer $k$   such that  $2 \leq k < p$.
\end{enumerate}
\end{theorem}

\noindent\textit{Proof.}
Let $Y$ be a complete intersection of multidegree 
$$\underline{d}=
(\underbrace{p,. \dots . ,p}_\text{$N-m$})$$ and 
let $i: Y \hookrightarrow \mathbb P^N$ be its embedding 
into $\mathbb P^N$. 
Let $k$ be an integer such that $2 \leq k < p$
and let 
$\pi: X \longrightarrow Y$ be a double cover of $Y$
branched along a smooth member 
of $|\mathcal O_Y(2k)|$. Then 
the hypothesis of Theorem~\ref{theorem.embedding}  (b)
are satisfied,
so $i \circ \pi$ can be deformed to an embedding. 
Let $X'$ the embedded variety. The subvariety $X'$ has degree 
$2p^{N-m}$ and 
\begin{equation*}
\omega_X=\mathcal O_X(p(N-m)+k-N-1). 
    \end{equation*}
If $X'$ were a complete intersection, of multidegree 
$(d_1',d_2',\dots, d_{N-m}')$, then Proposition~\ref{nci} would imply
\begin{equation*}
\sum d_i'= p(N-m)+k, \ \textrm{and} \ 
\prod d_i'= 2p^{N-m}
\end{equation*}
The second equality implies either $d_1'=2$, $d_2'=\cdots = d_{N-m-1}'=p$, $d_{N-m}'=p^2$ or 
$d_1'=\cdots = d_{N-m-1}'=p$, $d_{N-m}'=2p$. 
In the first case, $\sum d_i'= p(N-m)+k$ would be equivalent to $p(p-2)=k-2$, so either $k=2$ and $p=2$ or 
$k >2$ and $p$ divides $k-2$ and both contradict our 
hypothesis. 
In the second case, $\sum d_i'= p(N-m)+k$ would be equivalent to $p=k$, which again contradicts our 
hypothesis. Then $X'$ is not a complete intersection. 
The claim about complete linear series follows from 
Remark~\ref{cls}. \QEDB

\begin{theorem}\label{main0.5}
Let $m$ and $N$ be any integers such that $3 \leq m < \frac{N-1}{2}$. 
\begin{enumerate}
    \item  There exist smooth, non--complete intersection subvarieties $X_m'$, of dimension $m$ and degree 
   $(N-2m+1)(2(N-2m))^{N-m}$,
    embedded in $\mathbb P^N$ by a complete linear series.
   \item A subvariety $X'$ as in (1) is
 $s$--subcanonical, where $s=2(N-m+1)(N-2m)-N-1$.
\end{enumerate}
\end{theorem}

\color{black}

\noindent\textit{Proof.} Fix an integer $m\geq 3$. 
\textcolor{black}{Let $n=N-2m+1$. Note that $n \geq 3$.}
Let $Y_m$ be a complete intersection of multidegree 
$$\underline{d}=
(\underbrace{2(n-1),\cdots\cdots,2(n-1)}_\text{$m+n-1$})$$ 
inside $\mathbb{P}^{\textcolor{black}{N}}$, let 
$\pi_m:X_m\to Y_m$ be an $n$ \textcolor{black}{simple} cyclic \textcolor{black}{cover} branched 
along a smooth member of $|\mathcal{O}_Y(2n)|$ \textcolor{black}{and let 
$i_m:Y_m\hookrightarrow\mathbb{P}^{\textcolor{black}{N}}$ be the embedding
of $Y_m$ in $\mathbb P^N$}. 
We know by Theorem ~\ref{theorem.embedding}  (b) that 
$i_m\circ\pi_m$ deforms to an embedding. Let the embedded variety be $X_m'$. Assume $X_m'$ is a complete intersection of multidegree $(d_1',d_2',\cdots,d_{m+n-1}')$. Then, by Proposition ~\ref{nci}, we know that $$\sum d_i'=2(n-1)(m+n)
\textrm{, and }\prod d_i'=n((2(n-1))^{m+n-1}.$$ 
By arithmetic mean--geometric mean inequality, we know that $$\left(\frac{\sum d_i'}{m+n-1}\right)^{m+n-1}\geq \prod d_i'$$ 
$$\implies\left(\dfrac{m+n}{m+n-1}\right)^{m+n-1}=
\left(1+\dfrac{1}{m+n-1}\right)^{m+n-1}\geq n,$$
which is a contradiction since 
$(1+\frac{1}{m+n-1})^{m+n-1}$ is an increasing function 
of $m+n$ and the limit at infinity is $e$, but $n\geq 3$. 
The claim about complete linear series follows from Remark ~\ref{cls}. \QEDB\par

\begin{remark}\label{remark.moduli2} Arguing as in Remark~\ref{remark.moduli} we conclude that:
\begin{enumerate}
    \item Theorem~\ref{theorem.embedding.b} implies the existence, for fixed $m, N, n$ with $m \geq 3, n \geq 2, N \geq 2m+n-1$,  of infinitely many different moduli spaces
with the properties of the ones of Remark~\ref{remark.moduli}.
\item For $m, n$ fixed with $m \geq 3, n \geq 2$ or, simply, for $m$ fixed with $m \geq 3$, there exist smooth varieties of dimension $m$ as in Theorem~\ref{theorem.embedding.b} with $p_g$ arbitrarily large.
\item Theorem~\ref{real.main0.5} implies, for fixed $m, N$ with $3 \leq m < N/2$, the existence of non--complete intersection subvarieties of dimension $m$ in $\mathbb P^N$ belonging to infinitely many moduli spaces with the properties of the ones of Remark~\ref{remark.moduli}.
\item For $m$ fixed with $m \geq 3$  there exist smooth non--complete intersection subvarieties of dimension $m$ as in Theorem~\ref{real.main0.5} and Theorem~\ref{main0.5} with $p_g$ arbitrarily large.
\end{enumerate}
\end{remark}

\section{\textcolor{black}{Deforming finite morphisms
to non--embeddings }}\label{section.cyclic.covers.non.embeddings}

\subsection{Varieties with birational morphisms.}\label{birsubc} Now we study the cases for which Theorem ~\ref{birational} is applicable \textcolor{black}{to simple cyclic covers $\pi$}. Recall that, in this case $\varphi$ \textcolor{black}{can be deformed}  to a {birational} morphism \textcolor{black}{$\varphi'$}, which a priori, is not an embedding. 
\color{black}
We will show that these birational morphisms $\varphi': X' \longrightarrow \mathbb P^N$ exist for any $N$ and for varieties $X'$ of any dimension $m$, $3 \leq m \leq N-1$.  Some of these morphisms $\varphi'$ are strictly birational. \color{black}
Although we will mostly be interested in producing birational subcanonical morphisms, we start by showing an example of birational morphisms induced by  a non--complete linear series:  

\begin{example}\label{example.non.complete.series.birational}
Set $k=1$, and assume $N\leq \frac{nm+s}{n-1}$. 
Assume the notation of Example~\ref{example.non.complete.series.embedding} regarding $X, Y$  and $\pi$. 
 Furthermore, assume that one of the following conditions hold; 
    \begin{equation}\label{ex1}
        \floor{\frac{n}{2}}+m\leq N\leq \textrm{min}\bigg\{\frac{nm+2+s}{n-1}-1,2m+s+2-n\bigg\}\textrm{, or, } 
    \end{equation}
    \begin{equation}\label{ex1'}
        \floor{\frac{n}{2}}+m+1\leq N\leq 2m+s+2-n.
    \end{equation}
    Then the general deformation of 
$\varphi$  is a morphism induced by a non--complete linear series, 
finite and 
    birational 
    onto its image. Indeed, ~\eqref{ex1} guarantees that $\alpha\geq n-1$ and ~\eqref{ex1'} 
    guarantees that $N-m-1\geq\floor{n/2}$ so that Theorem ~\ref{birational}  applies.
\end{example}

In the remaining of this subsection, we focus in producing birational morphisms induced by complete linear series. We start with a result that proves some numerical inequalities.

\color{black}

\begin{proposition}\label{lemma.birational}
In the situation of Set-up ~\ref{subc}, assume $\varphi$ is $s$--subcanonical. If the hypothesis of Theorem ~\ref{birational}  holds, then,
\begin{equation}\label{eq.birational}
 m+\floor{n/2}\leq N\leq 2(m+\floor{n/2})-2(n-1)(\floor{n/2}+1)+s+1  \textcolor{black}{
 = 2m-2((n+2)\floor{n/2} -n) + s + 3}
\end{equation}
\textcolor{black}{and, consequently,} $$s\geq
\textcolor{black}{(2n+5)\floor{n/2} -2n-m-3.}$$
\color{black}
In particular, 
\begin{equation}\label{eq.birational.1}
    m+1 \leq N \leq 2m + s -1.
\end{equation}
\end{proposition}

\color{black}
\noindent\textit{Proof.} Assume the hypothesis of Theorem ~\ref{birational}  holds. Then $r=N-m\geq\floor{n/2}$ and that gives the lower bound.\par 
Now, under the assumption, $\delta$ will be least if the \textit{unordered} multidegree is the following;
$$\underline{d}_{\textrm{unord}}=(\underbrace{k(n-1),\cdots\cdots,k(n-1)}_\text{$\floor{n/2}$},\beta_1,\cdots\cdots,\beta_{N-m-\floor{n/2}}),$$ where $\beta_i\geq 2$ for all $i$. Since $\delta+k(n-1)=N+s+1$, we obtain;
$$N+s+1\geq \floor{n/2}k(n-1)+2(N-m-\floor{n/2})+k(n-1).$$
An elementary computation completes the proof since $k\geq 2$.\QEDB\par

\color{black}

\medskip

\medskip

Now we see that, for all set of invariants satisfying 
\eqref{eq.birational} and 
\eqref{eq.birational.1}, there exist morphisms $\varphi$ to which Theorem~\ref{birational}  can be applied, therefore 
producing infinitely many birational morphisms  in the ranges \eqref{eq.birational} and 
\eqref{eq.birational.1}.

\begin{theorem}\label{prop.birational}
Given any integers
$n, m, s$ and $N$ such that  $m  \geq 3$ 
and $n \geq 2$ and \eqref{eq.birational} holds,
then there 
exist 
smooth varieties $X'$ of dimension $m$ and 
$s$-subcanonical birational morphisms $\varphi':X' \longrightarrow \mathbb P^N$  which are 
deformations of  morphisms $\varphi$, where
$\varphi, n, m$ and $N$ are as in Set-up~\ref{subc} and $\varphi$ satisfies the hypothesis of Theorem~\ref{birational}. In particular, there are $X'$ and $\varphi'$ as above for $m, s$ and $N$ satisfying
\eqref{eq.birational.1}.
\end{theorem}
\color{black}

\noindent\textit{Proof.}
Assume $m\geq 3$, $n\geq 2$ and \eqref{eq.birational} holds.
\color{black}
For any integer $k$ such that 
\begin{equation*}
2 \leq     k \leq \frac{2(m + \lfloor{n/2}\textcolor{black}{\rfloor})+s+1-N}{(\lfloor{n/2}\textcolor{black}{\rfloor}+1)(n-1)},
\end{equation*}
(e.g., $k=2$), 
\color{black}
 there are integers 
$\beta_1,\beta_2,\cdots,\beta_{N-m-\floor{n/2}}\geq 2$, 
and $\alpha_1,\dots\alpha_{\floor{n/2}}\geq \textcolor{black}{k}(n-1)$ 
satisfying the following equation;
$$\sum\beta_i+\sum\alpha_i+\textcolor{black}{k}(n-1) 
=N+s+1.$$ For any such choices of $\alpha_i$'s, 
and $\beta_j$'s, 
let $Y$ be a complete intersection in $\mathbb{P}^N$ of multidegree $$\underline{d}_{\textrm{unord}}=(\alpha_1,\cdots\cdots\alpha_{\floor{n/2}},
\beta_1,\cdots,\beta_{N-m-\floor{n/2}}).$$
Let $\pi:X\to Y$ be a simple cyclic cover branched 
along a smooth member of $|\mathcal{O}_Y(\textcolor{black}{k}n)|$. 
Then $\varphi$ is an $s$--subcanonical cover of degree $n$ satisfying the hypothesis of Theorem ~\ref{birational}. Thus, a general deformation of $\varphi$ is {birational} onto its image.\QEDB 

\begin{example}\label{enr} The following table describes \textcolor{black}{the invariants of} the  first few varieties \textcolor{black}{$X'$ and birational morphisms $\varphi'$ of Theorem~\ref{prop.birational}, for} $k =2$ \textcolor{black}{and $n=3, 4$}.  \textcolor{black}{Proposition ~\ref{nci} and Corollary ~\ref{kpr} imply that,} for light blue  rows $\varphi'$ is not an embedding 
while, 
for the white rows, 
\textcolor{black}{$\varphi'$ would not be an embedding {if Hartshorne's conjecture is true.}} 

\vspace{5pt}

\begin{center}
 \begin{tabular}{c|c|c|c|c|c|c|c} 
 \hline
 $m$ & $n$ & $k$ & $N$ & $s$ & $\underline{d}$ & \color{black}$K_{X'}^m$ & \color{black}$p_g(X')$\\ 
 \hline\hline
  \rowcolor{LightCyan}
 $17$ & $4$ & $2$ & $19$ & $-1$ & $(6,7)$ & $-168$ & $0$\\
 \hline
 $17$ & $4$ & $2$ & $20$ & $-1$ & $(2,6,6)$ & $-288$ & $0$\\
 \hline
 \rowcolor{LightCyan}
 $16$ & $4$ & $2$ & $18$ & $0$ & $(6,7)$ & $0$ & $1$\\
 \hline
 $16$ & $4$ & $2$ & $19$ & $0$ & $(2,6,6)$ & $0$ & $1$\\
 \hline
 \rowcolor{LightCyan}
 $15$ & $4$ & $2$ & $17$ & $1$ & $(6,7)$ & $168$ & $18$\\
 \hline
 $15$ & $4$ & $2$ & $18$ & $1$ & $(2,6,6)$ & $288$ & $19$\\
 \hline
 \end{tabular}
 \end{center}
\end{example}

\begin{remark} 
Arguing as in Remark~\ref{remark.moduli} we conclude that:
\begin{enumerate}
    \item Theorem~\ref{prop.birational} implies the existence, for fixed $m, N, n$ with $m \geq 3, n \geq 2, N \geq m+\floor{n/2}$,  of infinitely many different moduli spaces  having reduced and irreducible components with a locally closed locus that parametrizes smooth varieties with an $s$-subcanonical morphism, finite and of degree $n$ onto its image, whereas the general points of the components correspond to smooth varieties with an $s$-subcanonical morphism which is a finite birational morphism onto its image. In Section~\ref{section.moduli} we will describe more specifically this phenomenon in the case of the canonical map, i.e., if $s=1$ (see Corollary~\ref{def of can morphisms birational}).
    
\item For $m, n$ fixed with $m \geq 3, n \geq 2$ or, simply, for $m$ fixed with $m \geq 3$, there exist smooth varieties of dimension $m$ as in Theorem~\ref{prop.birational} with $p_g$ arbitrarily large.
\end{enumerate}
\end{remark}

\subsection{\texorpdfstring{Degree $n$ subcanonical morphisms whose
deformations are of degree $n$}{Lg}} 
 \textcolor{black}{We now} study the cases for 
which Proposition ~\ref{nonexistence} is applicable \textcolor{black}{to simple cyclic covers $\pi$}. In these cases 
 the degree of any 
 deformation 
of $\varphi$ in these cases remains {unchanged}.

\smallskip

\begin{proposition}\label{prop.degree.unchanged}
In the situation of Set-up ~\ref{subc}, assume $\varphi$ is $s$--subcanonical and the hypothesis of Proposition ~\ref{nonexistence} is satisfied. Then, 
\begin{itemize}
    \item[(1)] 
    \begin{equation}\label{eq.degree.unchanged}
   \mathrm{max}\bigg\{\frac{s+1+m(k-1)-k(n-1)}{k-2},m+1\bigg\}\leq N\leq 2m+s+1-k(n-1)     
    \end{equation}
    and 
    \begin{equation*}
      \textcolor{black}{3 \leq k \leq \frac{m+s}{n-1}.}
    \end{equation*}

\item[(2)] If $k=3$, then $N=2m+s+1-3(n-1)$ and $\underline{d}=(2,\cdots\cdots,2)$.
\end{itemize}
\end{proposition}
\noindent\textit{Proof.} 
Since $\delta+k(n-1)=N+s+1$ and $\delta\geq 2(N-m)$, 
the upper bound follows.\par 
Since $\underline{d}=(d_1,\cdots,d_r)$ and 
$d_i\leq k-1$, it follows that $\delta\leq (N-m)(k-1)$. 
Consequently, $$N+s+1\leq (N-m)(k-1)+k(n-1).$$ 
An easy computation completes the proof of (1), 
and (2) is a consequence of (1).\QEDB\par
\vspace{5pt}

\color{black}
\begin{remark}\label{rmk.degree.unchanged}
Let 
\begin{equation*}
    \kappa=\floor{\frac{m+s}{n-1}}.
\end{equation*}
\textcolor{black}{The function 
on $k$, $k \geq 3$,
\begin{equation*}
    \frac{k(m-n+1)+s-m+1}{k-2}
\end{equation*}
is strictly decreasing.
}
Then, with the hypothesis of Proposition~\ref{prop.degree.unchanged}, 
\textcolor{black}{\eqref{eq.degree.unchanged} implies}
\begin{equation}\label{eq.degree.unchanged.2}
   \mathrm{max}\bigg\{\frac{\kappa(m-n+1)+s-m+1}{\kappa-2},m+1\bigg\}\leq N \leq 2m +s-3n+4; 
\end{equation}
in particular
\begin{equation*}
    s \geq 3n-m-3.
\end{equation*}
\end{remark}

\color{black}

We show now that, for invariants satisfying \eqref{eq.degree.unchanged}
 and \eqref{eq.degree.unchanged.2},
 there are 
morphisms $\varphi$ satisfying the hypothesis of 
Proposition~\ref{nonexistence}. Thus, the degree of  these morphisms $\varphi$  remains constant under deformation. 

\begin{proposition}\label{prop.degree.unchanged.2}
Let $k,m,n, s, \textcolor{black}{N}$ be integers satisfying $k\geq 3$, $m\geq 3$ 
and $n\geq 2$ such that \eqref{eq.degree.unchanged} holds \textcolor{black}{or let 
$m,n, s, \textcolor{black}{N}$ be integers satisfying  $m\geq 3$ 
and $n\geq 2$ such that \eqref{eq.degree.unchanged.2} holds.}
Then there are $s$-subcanonical morphisms  $\varphi$, where
$\varphi, n, m$ and $N$ are as in Set-up~\ref{subc}, all whose deformations have degree $n$ onto their image. 
\end{proposition}

\color{black}

\begin{proof}
\textcolor{black}{Let $m, n, s , N$ integers satisfying $m\geq 3$ 
and $n\geq 2$ such that \eqref{eq.degree.unchanged.2} holds. Then we can choose an integer $k$, with $3 \leq k \leq \kappa$, such that $k,m,n, s, N$ satisfy \eqref{eq.degree.unchanged}.}
Then, there 
are integers $\beta_r\geq\cdots\geq\beta_1\geq 2$ 
such that $\beta_i\leq k-1$ and 
$\sum\beta_i+k(n-1)=N+s+1$. Let $Y$ be of 
multidegree $\underline{d}=(\beta_1,\cdots,\beta_r)$ 
inside $\mathbb{P}^N$. Let $\pi:X\to Y$ be a 
simple cyclic $n$ cover branched along a smooth 
divisor in $|\mathcal{O}_Y(nk)|$. Then $\varphi$ 
is $s$--subcanonical and  satisfies the hypothesis 
of Proposition ~\ref{nonexistence}, 
consequently any deformation of 
$\varphi$ has degree $n$.
\end{proof}

\section{
Deformations of finite morphisms: the case of $\mathbb{Z}_{n/2}\times\mathbb{Z}_2$}\label{secz}

In this section, we shift the focus to \textcolor{black}{the study of deformations of morphisms which factor through} iterations of simple cyclic covers. More precisely, we will deform $\mathbb{Z}_{n/2}\times\mathbb{Z}_2$ covers
to produce smooth subvarieties embedded by complete linear series. Even though we will only  explicitly exhibit the invariants of subvarieties of small codimension, using Theorems~\ref{birational} and \ref{theorem.embedding} as we did in Sections~\ref{4}, \ref{section.cyclic.covers.nci} and \ref{section.cyclic.covers.non.embeddings}, one can easily produce
\begin{enumerate}
\item for any $m, N$ such that $3 \leq m \leq N-1$, 
$m$-dimensional smooth subvarieties in $\mathbb P^N$ of infinitely many different degrees (see Theorem~\ref{prop.embedding.bicyclic}); 
    \item smooth subvarieties 
    in $\mathbb P^N$ of the dimension $m$ in the range $3 \leq m < N/2$
    \textcolor{black}{which are}  not complete intersections;
    \item smooth varieties \textcolor{black}{equipped} with birational subcanonical morphisms which are not embeddings; 
    \item smooth varieties \textcolor{black}{equipped} with birational subcanonical morphisms, which \textcolor{black}{would not be} complete intersections {if} they \textcolor{black}{were} embeddings. 
\end{enumerate} 
\color{black} Furthermore, applying  Theorem ~\ref{2:1} we will show the existence of   $\mathbb{Z}_{n/2}\times\mathbb{Z}_2$ covers whose degree  under deformation drops to half. We remark that we
have not applied Theorem~\ref{2:1} in the previous sections devoted to simple cyclic covers. Although in this section we only deal with 
$\mathbb{Z}_{n/2}\times\mathbb{Z}_2$ covers, our arguments show how to proceed for more general iterations of simple cyclic covers \textcolor{black}{and also for simple cyclic covers of even degree}.

\vspace{5pt}

We describe now our set-up for this section.

\begin{set-up}\label{znz2}
Let \textcolor{black}{$X, Y, m, n$ and $\pi$} be as in Set-up ~\ref{setup1} (1). 
Let $p_1:X_1\to Y$ be a double cover branched along 
a smooth divisor $D_2$ in $|\mathcal{O}_Y(\textcolor{black}{2}l)|$ 
for some $l\in\mathbb{Z}_{>0}$. \color{black} Let $n$ be an even integer, $n \geq 4$, and 
\color{black}
let $p_2:X_2\to Y$ 
be a simple cyclic cover of degree \textcolor{black}{$n/2$,} branched 
along a smooth divisor $D_1$ in 
$|\mathcal{O}_Y\textcolor{black}{(\frac{n}{2}k)}|$ for some 
$k\in\mathbb{Z}_{>0}$. 
Assume $D_1$ and $D_2$ intersect transversally. 
 \textcolor{black}{Assume} $X:=X_1\times_Y X_2$ and  $D_1$ and $D_2$ intersect transversally  \textcolor{black}{(therefore, $X$ is smooth)} and 
\textcolor{black}{assume} $\pi:X\to Y$ 
\textcolor{black}{is}  the natural morphism from the fiber product to $Y$.
\end{set-up}

\begin{remark}\label{remark.subcanonical.ZnZ2}
In the situation of Set-up ~\ref{znz2},  $K_X=\pi^*\mathcal{O}_Y(-N-1+\delta+l+k\textcolor{black}{(\frac{n}{2}-1)})$. Consequently, $(X,L)$ is $s$--subcanonical if and only if $\delta+l+\textcolor{black}{(\frac{n}{2}-1)}k=N+s+1$. \textcolor{black}{From Remark~\ref{remark.subcanonical}, we get} the following facts;
\begin{itemize}
    \item[$(1)$] $X$ is a Fano variety if and only if $\delta+l+\textcolor{black}{(\frac{n}{2}-1)}k\leq N$ (in this case, $Y$ is also Fano). $(X,L)$ is a Fano polarized variety of index $-s$ if and only if $N+1+s=\delta+l+\textcolor{black}{(\frac{n}{2}-1)}k$.
    \item[$(2)$] $X$ is a Calabi--Yau variety if and only if $N+1=\delta+l+\textcolor{black}{(\frac{n}{2}-1)}k$ (in this case, $Y$ is Fano).
    \item[$(3)$] $X$ is a variety of general type if and only if $\delta+l+\textcolor{black}{(\frac{n}{2}-1)}k\geq N+2$. The morphism $\varphi$ (respectively $(X,L)$)  is:
    \begin{itemize}
        \item[(a)] Canonical if and only if $\delta+l+(n-1)k=N+2$ and $k,l\geq 2$ (resp. $\delta+l+\textcolor{black}{(\frac{n}{2}-1)}k=N+2$); in this case $Y$ is Fano.
        \item[(b)] Subcanonical if and only if $\delta+l+\textcolor{black}{(\frac{n}{2}-1)}k\geq N+3$ and $k,l\geq 2$ (resp. $\delta+l+\textcolor{black}{(\frac{n}{2}-1)}k\geq N+3$).
    \end{itemize}
\end{itemize}
\end{remark}

\subsection{Subvarieties with small codimension}\label{embznz2} We study the cases in which Theorem ~\ref{theorem.embedding}  (a) applies. The following \textcolor{black}{proposition, whose proof we omit,} is the analogue of \textcolor{black}{Proposition~\ref{lemmaemb}.} 

\begin{proposition}\label{lemmaemb2}
In the situation of Set-up ~\ref{znz2}, \textcolor{black}{if} 
$\varphi$ is $s$--subcanonical and the hypothesis of Theorem ~\ref{theorem.embedding} (a) holds, then,
\begin{equation}\label{t1}
    m+n-1\leq N\leq 2(m+n-1)-n(\textcolor{black}{n/2}+1)+s+1 
    \textcolor{black}{=2m-n(n/2-1)+s-1}
\end{equation}
and, \textcolor{black}{therefore}, 
\begin{equation*}
 s\geq \textcolor{black}{\frac{n^2}{2}}-m.   
\end{equation*}
Further, if $s=\textcolor{black}{{n^2}/{2}}-m$, then $N=m+n-1$ and the 
\emph{unordered} multidegree of $Y$ is 
$$\underline{d}_{\textrm{unord}}=
(2,4,\cdots,\textcolor{black}{n-2},2,4,\cdots,n).$$
\end{proposition}

\begin{theorem}\label{prop.embedding.bicyclic}
Given any integers
$n, m, s$ and $N$ such that  $m  \geq 3$, 
$n$ even, and $n \geq 4$,  if  \eqref{t1} holds,
then there 
exist 
smooth varieties $X'$ of dimension $m$ and 
$s$-subcanonical embeddings $\varphi':X' \longrightarrow \mathbb P^N$  such that 
\begin{enumerate}
    \item[(a)] the morphisms $\varphi'$ are 
    deformations of  morphisms $\varphi$, where 
$\varphi, m$ and $N$ are as in Set-up~\ref{subc} and $\varphi$ satisfies the hypothesis of Theorem~\ref{theorem.embedding} (a); 
\item[(b)] the subvarieties $\varphi'(X')$ are one-parameter deformations, as described in Theorem~\ref{theorem.embedding}, of  multiplicity $n$ rope subschemes.
\end{enumerate}
For any given integers
$n, m, N$ and $s$ satisfying   $m  \geq 3$, $n$ even,
$n \geq 4$ and \eqref{t1}, there are infinitely many non--isomorphic subvarieties $\varphi'(X')$ as above. 

\end{theorem}
\color{black}

\begin{proof}
 
\color{black}
Let $m,n, s \in\mathbb{Z}$, with $m\geq 3$, $n$ even, $n\geq 4$ satisfying 
\eqref{t1}.
For any integer $k,l\geq 2$ such that 
\begin{equation*}
     k(n^2/4-1)+l(n/2)\leq 2(m+n)+s-N-1,
\end{equation*}
then there are integers $\beta_1,\beta_2,\cdots,
\beta_{N-m-n+1}\geq 2$ satisfying the equation
$$\sum\beta_i+k\left(\frac{n^2}{4}-1\right)+\frac{n}{2}l=N+s+1.$$ For any such choices 
of $\beta_i$'s, let $Y$ be a complete intersection 
in $\mathbb{P}^N$ of multidegree 
$$\underline{d}_{\textrm{unord}}=(k,2k,\cdots\cdots,(n/2-1)k,\textcolor{black}{l},l+k,l+2k,\cdots,l+k(n/2-1),\beta_1,\cdots,\beta_{N-m-n+1}).$$
Let $\pi:X\to Y$. 
Let $\varphi:X\to Y$ be the natural morphism from the fiber product $X:=X_1\times_Y X_2$ of a double cover $p_1:X_1\to Y$ branched along a smooth member $D_2$ of $|\mathcal{O}_Y(2l)|$, and a simple cyclic cover $p_2:X_2\to Y$ of degree $n/2$, branched along a smooth member $D_1$ of $|\mathcal{O}_Y((nk/2)|$ such that $D_1$ and $D_2$ intersect transversally. Then $X$ is smooth and $\varphi$ is an $s$--subcanonical cover of degree $n$ satisfying the hypothesis of Theorem ~\ref{theorem.embedding}  (a). Thus, a general deformation of $\varphi$ is an {embedding}.
\end{proof}

\color{black}
\begin{corollary}\label{cor.embedding.bicyclic}
Given any integers $m$ and $N$ such that $3 \leq m \leq N-3$, 
there 
exist 
smooth varieties $X'$ of dimension $m$ and embeddings $\varphi':X' \longrightarrow \mathbb P^N$, with  $\varphi'(X')$  having infinitely many different degrees, such that 
\begin{enumerate}
    \item[(a)] the morphisms $\varphi'$ are 
    deformations of  morphisms $\varphi$, where 
$\varphi, m$ and $N$ are as in Set-up~\ref{znz2}, $\varphi$ is a $\mathbb{Z}_{n/2}\times \mathbb{Z}_2$ cover with even $n\geq 4$ and it satisfies the hypothesis of Theorem~\ref{theorem.embedding} (a); 
\item[(b)] the subvarieties $\varphi'(X')$ are one-parameter deformations, as described in Theorem~\ref{theorem.embedding}, of  rope subschemes.
\end{enumerate}
 More precisely, in the above cases we have $s\geq N-2m+5$ and the above rope subschemes can be chosen to be non--complete intersections.
 
\end{corollary}

\subsection{Varieties with degree \textcolor{black}{\texorpdfstring{$n/2$}{Lg}} subcanonical morphisms}\label{subsection.half} Finally, we study the cases for which Theorem ~\ref{2:1} applies. This is a new case that did not appear in the previous section. 

\medskip

\begin{proposition}\label{prop.bicyclic.2.1-Javier}
In the situation of Set-up ~\ref{znz2},
let 
\begin{equation*}
    \kappa= \left \lfloor \frac{2(m+s-1)}{n-2} \right \rfloor \ \textrm{and} \ 
    \kappa'= \left \lfloor \frac{2(m+s-2)}{n-2} \right \rfloor,
\end{equation*}
let 
 $\varphi$ be $s$--subcanonical \textcolor{black}{and assume hypothesis (1) 
and (2) of Theorem ~\ref{2:1} are satisfied.} 
 \smallskip
 
 \begin{enumerate}
     \item \textcolor{black}{If $l=1$},
then
 \begin{equation}\label{eq.prop.bicyclic.2.1.Javier.first}
  \mathrm{max}\bigg\{m+1, \frac{(m-n/2+1)\kappa +s-m}{\kappa-2}\bigg\}\leq N\leq 2m -3n/2+s+3; 
\end{equation}
in particular, 
\begin{equation}\label{eq.prop.bicyclic.2.1.Javier.second}
    s \geq 3n/2 - m -2.
    \end{equation}
    
\item  If  hypothesis (1') of Theorem ~\ref{2:1} is satisfied, then 
\begin{equation}\label{eq.prop.bicyclic.2.1.Javier.third}
\mathrm{max}\bigg\{m+1,\frac{(m-n/2+2)\kappa' +s-m+1}{\kappa'-2}\bigg\} \leq N\leq 2m -5n/2+s+4; 
\end{equation}
in particular, 
\begin{equation}\label{eq.prop.bicyclic.2.1.Javier.fourth}
    s \geq 5n/2 - m -3.
    \end{equation}
    
\item  If  $l \geq 2$, then 
    \begin{equation}\label{eq.prop.bicyclic.2.1.Javier.third.bis}
\textcolor{black}{\mathrm{max}\bigg\{m+1,\frac{(m-n/2+1/2)\kappa' +s-m+3/2}{\kappa'-2}\bigg\}}\leq N\leq 2m -5n/2+s+4;
  \end{equation}
  in particular, 
\begin{equation}\label{eq.prop.bicyclic.2.1.Javier.fourthbis}
    s \geq 5n/2 - m - \textcolor{black}{3}.
    \end{equation}
 \end{enumerate}
 \end{proposition}

\begin{proof}
\textcolor{black}{We prove (1) first, so we} assume $l=1$. 
By Remark~\ref{remark.subcanonical.ZnZ2} and Theorem~\ref{2:1} (2),
\begin{equation}\label{eq.prop.bicyclic.2.1.Javier.sixth}
    \textrm{max}\bigg\{m+1, \frac{(m-n/2+1)k+s-m}{k-2}\bigg\} \leq N \leq 2m+s -(n/2-1)k.  
\end{equation}
The function 
\begin{equation*}
   \frac{(m-n/2+1)k+s-m}{k-2} 
\end{equation*}
on $k$, $k \geq 3$, is strictly decreasing.
Since
\eqref{eq.prop.bicyclic.2.1.Javier.sixth} implies
\begin{equation*}
    k \leq \frac{2(m+s-1)}{n-2},
\end{equation*}
the minimum value of 
\begin{equation*}
   \frac{(m-n/2+1)k+s-m}{k-2} 
\end{equation*}
is attained at $k=\kappa$. This proves 
\eqref{eq.prop.bicyclic.2.1.Javier.first} and \eqref{eq.prop.bicyclic.2.1.Javier.first} implies \eqref{eq.prop.bicyclic.2.1.Javier.second}.

 \smallskip
\textcolor{black}{Now we prove (2).} \textcolor{black}{Recall that} hypothesis (1') of Theorem ~\ref{2:1} is satisfied \textcolor{black}{and, in particular, $l \geq 2$}.
 By Remark~\ref{remark.subcanonical.ZnZ2},
 \begin{equation}\label{eq.prop.bicyclic.2.1.Javier.sixbis}
 \textrm{max}\bigg\{m+1, \frac{(m-n/2+2)k-2l+s-m}{k-2}\bigg\} \leq N \leq
 \textcolor{black}{2m+s-(n/2-1)k-2l +3},  
\end{equation}
and, by hypothesis (2) of Theorem ~\ref{2:1}, 
 \begin{equation}\label{eq.prop.bicyclic.2.1.Javier.seventh}
 \textrm{max}\bigg\{m+1, \frac{(m-n/2+1)k+s-m+1}{k-2}\bigg\} \leq N \leq 2m+s -(n/2-1)k\textcolor{black}{-1}.  
\end{equation}
The function 
\begin{equation*}
   \frac{(m-n/2+1)k+s-m+1}{k-2} 
\end{equation*}
on $k$, $k \geq 3$, is strictly decreasing. 
Since
\eqref{eq.prop.bicyclic.2.1.Javier.seventh} implies
\begin{equation*}
    k \leq \frac{2(m+s-2)}{n-2},
\end{equation*}
the minimum value of 
\begin{equation*}
   \frac{(m-n/2+1)k+s-m+1}{k-2} 
\end{equation*}
is attained at $k=\kappa'$.
This proves 
\eqref{eq.prop.bicyclic.2.1.Javier.third} and \eqref{eq.prop.bicyclic.2.1.Javier.third} implies \eqref{eq.prop.bicyclic.2.1.Javier.fourth}.

\smallskip
 
\textcolor{black}{Finally we prove (3), so we} now assume $l \geq 2$. Then \begin{equation*}
\textrm{max}\bigg\{m+1,\frac{(m-n/2+1)k-l+s-m+1}{k-2}\bigg\} 
\leq N \leq 2m + s -(n/2-1)k-2l + 3, 
\end{equation*}
\color{black}
and, by hypothesis (2) of Theorem ~\ref{2:1}, 
 \begin{equation*}
 \textrm{max}\bigg\{m+1, \frac{(m-n/2+1/2)k+s-m+3/2}{k-2}\bigg\} \leq N \leq 2m+s -(n/2-1)k-1.  
\end{equation*}
The function
\begin{equation*}
\frac{(m-n/2+1/2)k+s-m+3/2}{k-2}    
\end{equation*}
on $k$, $k \geq 3$, is strictly decreasing and, arguing as in the proof of (2), we see that its minimum is attained at $k=\kappa'$. Then 
\eqref{eq.prop.bicyclic.2.1.Javier.third.bis} holds and \eqref{eq.prop.bicyclic.2.1.Javier.third.bis} implies  \eqref{eq.prop.bicyclic.2.1.Javier.fourthbis}. 
\end{proof}

\begin{remark}
\begin{enumerate}
    \item If 
\begin{equation}\label{eq.remark.1}
  \kappa \geq \frac{2(m+s+2)}{n}
\end{equation}
(this happens for example if $s \geq n(n/4+1)-m-2$), then 
\begin{equation*}
  \frac{(m-n/2+1)\kappa +s-m}{\kappa-2} \leq m+1, 
\end{equation*}
so \eqref{eq.prop.bicyclic.2.1.Javier.first} becomes 
\begin{equation}\label{eq.remark.2}
m+1 \leq N\leq 2m -3n/2+s+3 
\end{equation}
in this case.
\item If 
\begin{equation}\label{eq.remark.3}
  \kappa' \geq \frac{2(m+s+3)}{n}
\end{equation}
(this happens for example if $s \geq n(n/4+2)-m-3/2$),
then 
\begin{equation*}
  \frac{(m-n/2+1)\kappa' +s-m+1}{\kappa'-2} \leq m+1, 
\end{equation*}
so \eqref{eq.prop.bicyclic.2.1.Javier.third} becomes 
\begin{equation}\label{eq.remark.4}
m+1 \leq N\leq 2m -5n/2+s+4
\end{equation}
in this case.
\end{enumerate}
\end{remark}

\begin{theorem}\label{new.degreehalf}
Let $m, n, s, N$ be integers, with $m \geq 3$ and $n \geq 4$, even.
\begin{enumerate}
    \item If \eqref{eq.prop.bicyclic.2.1.Javier.first} holds, or if
\eqref{eq.remark.1} and \eqref{eq.remark.2}
hold, then there 
exist 
smooth varieties $X$ of dimension $m$ and 
$s$-subcanonical morphisms $\varphi:X \longrightarrow \mathbb P^N$  such that a general deformation of $\varphi$ is a  morphism which is finite and of degree $n/2$ onto its image. 
\item If \eqref{eq.prop.bicyclic.2.1.Javier.third} holds,  or if
\eqref{eq.remark.3} and \eqref{eq.remark.4} hold, then there 
exist 
smooth varieties $X$ of dimension $m$ and 
$s$-subcanonical morphisms $\varphi:X \longrightarrow \mathbb P^N$  such that a general deformation of $\varphi$ is a flat morphism which is finite and of degree $n/2$ onto its image, which is smooth. 
\end{enumerate}
\end{theorem}

\begin{proof}
We first set $l=1$.
For any integers $k, N$ that satisfy
$k \geq 3$ and \eqref{eq.prop.bicyclic.2.1.Javier.sixth}, 
there are integers 
\begin{equation*}
   2 \leq \beta_1, \beta_2, \dots, \beta_{N-m} \leq k-1 
\end{equation*} 
such that 
\begin{equation*}
\sum\beta_i+1+k(n/2-1)=N+s+1.    
\end{equation*}
Then, for any such choices of $\beta_i$'s, let $Y$ be a smooth complete 
intersection of multidegree
$$\underline{d}_{\textrm{unord}}=(\beta_1,\cdots\cdots,\beta_{N-m}).$$ Let $\varphi:X\to Y$ be the morphism from the fiber product $X:=X_1\times_Y X_2$ of a double cover $p_1:X_1\to Y$ branched along a smooth member $D_2$ of $|\mathcal{O}_Y(2)|$, and a simple cyclic cover $p_2:X_2\to Y$ of degree {$n/2$}, branched along a smooth member $D_1$ of 
{$|\mathcal{O}_Y(nk/2)|$}, such that $D_1$ and $D_2$ intersect transversally. Then $X$ is smooth, $\varphi$ is an $s$--subcanonical cover of degree $n/2$ satisfying the hypothesis (1) and (2) of Theorem ~\ref{2:1}. Thus, a general deformation of $\varphi$ \textcolor{black}{is a morphism which is finite and} of degree $n/2$ \textcolor{black}{onto its image}. 
Therefore, if  $k, N$ are integers that satisfy
$k \geq 3$ and \eqref{eq.prop.bicyclic.2.1.Javier.sixth}, then there 
exist 
smooth varieties $X$ of dimension $m$ and 
$s$-subcanonical morphisms $\varphi:X \longrightarrow \mathbb P^N$  such that a general deformation of $\varphi$ is a morphism, which is finite and of degree $n/2$ onto its image.
This implies (1).

\smallskip

Now let us prove (2). For any integers $k, l, N$ that 
satisfy $l \geq 2$, $k \geq 2l+1$ and \eqref{eq.prop.bicyclic.2.1.Javier.sixbis}, 
there exist integers $2 \leq \beta_1, \beta_2, \dots, \beta_{N-m-1} \leq k-1$ such that 
\begin{equation*}
\sum\beta_i+2l+k(n/2-1)=N+s+1.    
\end{equation*}
Then, for any such choices of $\beta_i$'s, let $Y$ be a smooth complete 
intersection of multidegree
$$\underline{d}_{\textrm{unord}}=(l,\beta_1,\cdots\cdots,
\beta_{N-m-1}).$$
Let now $\varphi: X \longrightarrow Y$ be a morphism constructed as in the previous paragraph, except that this time $D_2$ is a smooth member of 
$|\mathcal O_Y(2l)|$. Then $X$ is smooth, $\varphi$ is an $s$--subcanonical cover of degree $n/2$ satisfying the hypothesis (1), (2) and (1') of Theorem ~\ref{2:1}. Thus, a general deformation of $\varphi$ is a flat morphism, which is finite and degree $n/2$ onto a smooth image. 
Therefore, if  $k, l, N$ are integers that 
satisfy $l \geq 2$, $k \geq 2l+1$ and \eqref{eq.prop.bicyclic.2.1.Javier.sixbis}, then there 
exist 
smooth varieties $X$ of dimension $m$ and 
$s$-subcanonical morphisms $\varphi:X \longrightarrow \mathbb P^N$  such that a general deformation of $\varphi$ is a flat morphism, which is finite and degree $n/2$ onto a smooth image.
This implies (2). 
\end{proof}

\color{black}
\begin{example}
\color{black}
The following table describes the  first few varieties \textcolor{black}{of codimension $2$ and} with $l=2$ and $k=5$ we obtain this way.
\vspace{5pt}

\begin{center}
 \begin{tabular}{c|c|c|c|c|c|c|c|c} 
 \hline
 $m$ & $n$ & $k$ & $l$ & $N$ & $s$ & $\underline{d}$ & $K_X^m$ & $p_g(X)$\\ 
 \hline\hline
 $9$ & $4$ & $5$ & $2$ & $11$ & $-1$ & $(2,2)$ & $-16$ & $0$\\
 \hline
 $14$ & $6$ & $5$ & $2$ & $16$ & $-1$ & $(2,2)$ & $-24$ & $0$\\
 \hline
 $8$ & $4$ & $5$ & $2$ & $10$ & $0$ & $(2,2)$ & $0$ & $1$\\
 \hline
 $13$ & $6$ & $5$ & $2$ & $15$ & $0$ & $(2,2)$ & $0$ & $1$\\
 \hline
 $7$ & $4$ & $5$ & $2$ & $9$ & $1$ & $(2,2)$ & $16$ & $10$\\
 \hline
 $12$ & $6$ & $5$ & $2$ & $14$ & $1$ & $(2,2)$ & $24$ & $15$\\
 \hline
 \end{tabular}
 \end{center}
\end{example}

\begin{remark}\label{remark.moduli3} 
Arguing as in Remark~\ref{remark.moduli} we conclude that:
\begin{enumerate}
    \item Theorem~\ref{new.degreehalf} implies, for fixed $m, N, n$ with $m \geq 3,  N \geq m+1, n \geq 4$,  $n$ even, the existence of infinitely many different moduli spaces  having reduced and irreducible components with a locally closed locus that parametrizes smooth varieties with an $s$-subcanonical morphism, finite and of degree $n$ onto a smooth image, whereas the general points of the components correspond to smooth varieties with an $s$-subcanonical morphism which is a morphism, which is finite and of degree $n/2$ onto its image, that can even be smooth. In Section~\ref{section.moduli} we will describe more specifically this phenomenon in the case of the canonical map, i.e., if $s=1$ (see Corollary~\ref{def of can morphisms halved}).
    
\item For $m, n$ fixed with $m \geq 3, n \geq 4$, $n$ even or, simply, for $m$ fixed with $m \geq 3$, there exist smooth varieties of dimension $m$ as in Theorem~\ref{new.degreehalf} with $p_g$ arbitrarily large.
\end{enumerate}
\end{remark}

\color{black}
\section{Deformations of finite morphisms: dihedral cover case}\label{5}

\color{black}

In this section, we  \textcolor{black}{look at} {\it non--abelian covers},  \textcolor{black}{specifically, at} simple dihedral covers.  With the aid of the results of Section ~\ref{3}, we will 
deform  \textcolor{black}{simple dihedral covers} to  \textcolor{black}{construct} small codimensional subvarieties  \textcolor{black}{embedded by complete linear series} inside projective space.  \textcolor{black}{As in the case of $\mathbb{Z}_{n/2}\times\mathbb{Z}_2$ covers,}  we will only \textcolor{black}{explicitly exhibit} small codimensional subvarieties.  However, using the results of Section ~\ref{3}, one can  produce
\color{black} $m$-dimensional smooth subvarieties in $\mathbb P^N$ of infinitely many different degrees for any $m, N$ such that $3 \leq m \leq N-1$;  smooth, non--complete intersection, $m$-dimensional subvarieties in $\mathbb P^N$, in the range $3 \leq m < N/2$; 
smooth varieties \textcolor{black}{equipped} with birational subcanonical morphisms which are not embeddings;
smooth varieties \textcolor{black}{equipped} with birational subcanonical morphisms, which \textcolor{black}{would not be} complete intersections {if} they  \textcolor{black}{were} embeddings; 
\textcolor{black}{and simple dihedral covers  whose degree under
deformation drops to half (by applying  Theorem~\ref{2:1} to the factorization that the index 2, cyclic subgroup of  $D_{n/2}$ induces on  a $D_{n/2}$  cover).}

\color{black}

\smallskip

Now we describe our set-up for \textcolor{black}{the} simple dihedral covers that we will deform.

\begin{set-up}\label{dn}
Let \textcolor{black}{$X, Y, \pi, n$ and $m$} be as in Set-up ~\ref{setup1} (1). Let \textcolor{black}{$n$ be even, $n \geq 6$ and let} $\pi:X\to Y$ be a smooth irreducible simple \textcolor{black}{$D_{n/2}$} cover associated to \textcolor{black}{sections $s_1 \in H^0(\mathscr{O}_Y(\frac{n}{2}k))$ and $s_2 \in H^0(\mathscr{O}_Y(2k))$ such that $s_1^2-s_2^{{n}/{2}}$ is smooth in the zero locus of $s_2 \neq 0$ and $s_1$ and $s_2$ intersect transversally  \textcolor{black}{(all this happens for instance for general choices of $s_1$ and $s_2$)}. }
\end{set-up}

\begin{remark}\label{remark.dihedral.subcanonical}
In the situation of Set-up ~\ref{dn}, $K_X=\pi^*\mathcal{O}_Y(\textcolor{black}{-N-1+\delta +\frac{n}{2}k})$. Consequently, $(X,L)$ is $s$--subcanonical if and only if $\delta+\textcolor{black}{\frac{n}{2}}k=N+s+1$. We also make a note of the following facts;
\begin{itemize}
    \item[$(1)$] \textcolor{black}{The variety} $X$ is a Fano variety if and only if $\delta+\textcolor{black}{\frac{n}{2}}k\leq N$ (in this case, $Y$ is also Fano). \textcolor{black}{If $\delta+{\frac{n}{2}}k\leq N$, then} $(X,L)$ is a Fano polarized variety of index $-s$ if and only if $N+1+s=\delta+\textcolor{black}{\frac{n}{2}}k$.
    \item[$(2)$] \textcolor{black}{The variety} $X$ is a Calabi--Yau variety if and only if $N+1=\delta+\textcolor{black}{\frac{n}{2}}k$ (in this case, $Y$ is Fano).
    \item[$(3)$] \textcolor{black}{The variety}  $X$ is a variety of general type if and only if $\delta+\textcolor{black}{\frac{n}{2}}k\geq N+2$. \textcolor{black}{Moreover} The morphism $\varphi$ (respectively $(X,L)$)  is
    canonical if and only if $\delta+\textcolor{black}{\frac{n}{2}}k=N+2$ and $k\geq 2$ (resp. $\delta+\textcolor{black}{\frac{n}{2}}k=N+2$); in this case $Y$ is Fano.
\end{itemize}
\end{remark}

We  
{{study}} the cases in which Theorem ~\ref{theorem.embedding}  (a) applies, and as before, we omit the proof of the following

\begin{proposition}\label{embdn}
\color{black}
In the situation of Set-up ~\ref{dn}, 
assume $\varphi$ is $s$--subcanonical. If the hypothesis of Theorem ~\ref{theorem.embedding}  (a) holds, then,
\begin{equation}\label{eq.embdn}
 m+n-1\leq N\leq
 \textcolor{black}{2(m+n-1) -n(n/2+1)+s +1}
\textcolor{black}{= 2m-n(n/2-1) + s-1}, 
\end{equation}
\textcolor{black}{so, in particular,} 
\begin{equation*}
 s\geq n^2/2-m.  
    \end{equation*}
Further, if $s=n^2/2-m$, then $N=m+n-1$ and the 
\emph{unordered} multidegree of $Y$ is 
$$\underline{d}_{\textrm{unord}}=
(2,2,4,4,\cdots,2(n/2-1),2(n/2-1),n).$$
\end{proposition}
\color{black}

\textcolor{black}{In the next theorem we show the existence of smooth subvarieties $\varphi'(X')$ obtained by deforming 
$s$-subcanonical morphisms $\varphi$ for which the inequalities of Proposition ~\ref{embdn} \textcolor{black}{hold}.}

\begin{theorem}\label{prop.embedding.dihedral}
Given any integers
$n, m, s$ and $N$ such that  $m\geq 3$, $n\geq 6$ even, and 
\eqref{eq.embdn} holds,
there 
exist 
smooth varieties $X'$ of dimension $m$ and 
$s$-subcanonical embeddings $\varphi':X' \longrightarrow \mathbb P^N$  such that 
\begin{enumerate}
    \item[(a)] the morphisms $\varphi'$ are 
    deformations of  morphisms $\varphi$, where 
$\varphi, m$ and $N$ are as in Set-up~\ref{dn} and $\varphi$ satisfies the hypothesis of Theorem~\ref{theorem.embedding} (a); 
\item[(b)] the subvarieties $\varphi'(X')$ are one-parameter deformations, as described in Theorem~\ref{theorem.embedding}, of  multiplicity $n$ rope subschemes.
\end{enumerate}
For any given integers
$n, m, N$ and $s$ satisfying   $m  \geq 3$, $n\geq 6$ even, 
satisfying \eqref{eq.embdn}, there are infinitely many non--isomorphic subvarieties $\varphi'(X')$ as above. 
\end{theorem}
\noindent\textit{Proof.} Under the assumption, there are integers $\beta_1,\beta_2,\cdots,\beta_{N-m-n+1}\geq 2$ satisfying the following equation;
$$\sum\beta_i+n^2/2=N+s+1.$$ For any such choices of $\beta_i$'s, let $Y$ be a complete intersection in $\mathbb{P}^N$ of multidegree $$\underline{d}_{\textrm{unord}}=(2,2,4,4,\cdots\cdots,2(n/2-1),2(n/2-1),n,\beta_1,\cdots\cdots,\beta_{N-m-n+1}).$$
Let $\varphi:X\to Y$ be a smooth simple dihedral cover associated to the line bundle $L=\mathcal{O}_Y(2)$ satisfying the conditions of Theorem ~\ref{cp}. Then $\varphi$ is an $s$--subcanonical cover of degree $n$ satisfying the hypothesis of Theorem ~\ref{theorem.embedding} (a). Thus, a general deformation of $\varphi$ is an 
{embedding}.\QEDB\par

\begin{example}\label{example.dihedral.1}
Now we describe \textcolor{black}{the invariants of} the first few 
\textcolor{black}{smooth, subvarieties $\varphi'(X')$ obtained when we deform $s$--subcanonical morphisms $\varphi$} as in Set-up ~\ref{dn}, \textcolor{black}{with $-1 \leq s \leq 1$,} for which the hypothesis of Theorem ~\ref{theorem.embedding} (a) holds \textcolor{black}{(this includes Calabi-Yau and canonically embedded subvarieties)}.  \textcolor{black}{It is interesting to note that} \textcolor{black}{the  subvarieties in rows 1, 2, 4, 5, 8 are \textcolor{black}{near the boundary of, but inside, the range of Hartshorne's conjecture. Precisely, in row 7, the codimension $r$ satisfies $r=(1/3)N-2$ and in rows 1, 2, 4, 5, 8 the codimension $r$ satisfy $(1/3)N-2 < N \leq (1/3)N-3$.}}

\vspace{5pt}

\begin{center}
 \begin{tabular}{c|c|c|c|c|c|c|c} 
 \hline
 $m$ & $n$ & $k$ & $N$ & $s$ & $\underline{d}$ & $K_{\textcolor{black}{X'}}^m$ & $p_g(\textcolor{black}{X'})$\\ 
 \hline\hline
 $20$ & $6$ & $2$ & $26$ & $-1$ & $(2,2,2,4,4,6)$ & $-4608$ & $0$\\
 \hline
 $19$ & $6$ & $2$ & 
 $24$ & $-1$ & $(2,2,4,4,6)$ & $-2304$ & $0$\\
 \hline
 $33$ & $8$ & $2$ & $40$ & $-1$ & $(2,2,4,4,6,6,8)$ & $-147456$ & $0$\\
 \hline
 $19$ & $6$ & $2$ & $25$ & $0$ & $(2,2,2,4,4,6)$ & $0$ & $1$\\
 \hline
 $18$ & $6$ & $2$ & $23$ & $0$ & $(2,2,4,4,6)$ & $0$ & $1$\\
 \hline
 $32$ & $8$ & $2$ & $39$ & $0$ & $(2,2,4,4,6,6,8)$ & $0$ & $1$\\
 \hline
 $18$ & $6$ & $2$ & $24$ & $1$ & $(2,2,2,4,4,6)$ & $4608$ & $25$\\
 \hline
 $17$ & $6$ & $2$ & $22$ & $1$ & $(2,2,4,4,6)$ & $2304$ & $23$\\
 \hline
 $31$ & $8$ & $2$ & $38$ & $1$ & $(2,2,4,4,6,6,8)$ & $147456$ & $39$\\
 \hline
 \end{tabular}
\end{center}
\end{example}
  
  \begin{example}\label{example.dihedral.2}
Theorem~\ref{prop.embedding.dihedral} allows also the construction of other smooth subvarieties closer to the boundary of the range of Hartshorne's conjecture, like the very small sample of the smooth subvarieties   in the range $r=(1/3)N-1$, displayed in Table \hyperref[t02]{2} of the introduction, which are  obtained by deforming $\varphi$, which factors through a  simple $D_{n/2}$ cover  of degree $n$, with $k=2$. It is therefore interesting to know whether these subvarieties and the ones of Example~\ref{example.dihedral.1} is a complete intersection, although we do not know whether the same is true for  some special deformations of $\varphi$ (see Question~\ref{question.ci.diherdral}).
 \end{example}

\begin{remark}\label{remark.moduli4} 
Arguing as in Remark~\ref{remark.moduli} we conclude that:
\begin{enumerate}
    \item Theorem~\ref{prop.embedding.dihedral} implies the existence, for fixed $m, N, n$ with $m \geq 3, n \geq 6$, $n$ even, and $N \geq m+n-1$,  of infinitely many different moduli spaces  having reduced and irreducible components with a locally closed locus that parametrizes smooth varieties with an $s$-subcanonical morphism, finite and of degree $n$ onto its image, whereas the general points of the components correspond to smooth varieties with an $s$-subcanonical morphism which is an embedding onto its image. In Section~\ref{section.moduli} we will describe more specifically this phenomenon in the case of the canonical map, i.e., if $s=1$ (see Corollary~\ref{def of can morphisms dn}).
    
\item For $m, n$ fixed with $m \geq 3, n \geq 6$, $n$ even or, simply, for $m$ fixed with $m \geq 3$, there exist smooth varieties of dimension $m$ as in Theorem~\ref{prop.birational} with $p_g$ arbitrarily large.
\end{enumerate}
\end{remark} 

\section{Moduli of varieties of general type}
\label{section.moduli}

\textcolor{black}{In this section we construct components of the moduli space of varieties of general type that are analogues of moduli space of curves with respect to the canonical map and its deformations. \textcolor{black}{More precisely, we show} the existence of a \textcolor{black}{locally closed locus} in \textcolor{black}{these} moduli components,  \textcolor{black}{which is analogous to the hyperelliptic locus of curves of genus bigger than $2$,} where the degree of the canonical map jumps up. \textcolor{black}{The deformation of the canonical maps} gives raise to one--parameter families where the special member is a rope of multiplicity $m\geq 2$ and the general member is a smooth canonically embedded subvariety. \textcolor{black}{An} interesting \textcolor{black}{point} to note is that \textcolor{black}{we find these canonically embedded varieties for any codimension $r$ and, in particular, for any $r$ in the range of the Hartshorne’s conjecture.}}
\textcolor{black}{The existence of the special loci \textcolor{black}{proved in this section is} but a particular case of \textcolor{black}{a} more general phenomenon. In fact, our results imply (see Remarks~\ref{remark.moduli}, \ref{remark.moduli2}, \ref{remark.moduli3}, \ref{remark.moduli4}) the existence of infinitely many irreducible components possessing loci which \textcolor{black}{correspond} to \textcolor{black}{the jumping} up of the degree of the subcanonical maps. \textcolor{black}{As it happens in this section,} the  deformations \textcolor{black}{of the canonical map in these cases give} rise to \textcolor{black}{analogous}, interesting one--parameter families. Because of the significance of the canonical map, \textcolor{black}{now} we focus and give the details only of the loci related to this map.}
\textcolor{black}{Finally, as explained in the introduction, in this section we also construct two distinct kinds of moduli components  which differ from the moduli space of curves.} \textcolor{black}{These are made precise in Corollaries \ref{def of can morphisms birational} and \ref{def of can morphisms halved} below.}

\begin{corollary}\label{def of can morphisms}
\color{black}
Let $m \geq 3$.  For each integer $n$ such that $2 \leq n \leq \sqrt{m+3}$, there is 
at least one reduced 
and irreducible component $\mathcal M_n$ of the moduli space of varieties of general type of dimension $m$ such that:
\begin{enumerate}
    \item[(i)] $\mathcal M_n$ has a locally closed locus that parametrizes varieties whose canonical map $\varphi$ is a morphism, finite of degree $n$ onto a smooth  image; more precisely, $\varphi$ factors through simple cyclic cover of degree $n$. 
    \item[(ii)] The general points of $\mathcal M_n$ 
    correspond to varieties whose canonical map is an embedding. 
    \item[(iii)] \textcolor{black}{The canonical models of the varieties corresponding to general points of $\mathscr{M}_n$ degenerate to a rope of multiplicity $n$, along a general one-parameter family that intersects the special locus \textcolor{black}{described in (i)}.}
\end{enumerate}  
For different values of $n$, the components $\mathcal M_n$ are different.
Letting $m$ vary, the general points of the $\mathcal M_n$ parametrize canonically embedded varieties of any codimension. 
\end{corollary}

\noindent\textit{Proof.}
\color{black}
Let $m \geq 3$. Let 
 $2 \leq n \leq \sqrt{m+3}$. There exists an integer $N$, 
 satisfying \eqref{eq.lemmaemb} for $s=1$. Then Theorem ~\ref{theorem.embedding.a} implies the existence of $\mathcal M_n$ satisfying (i) and (ii). Note that $N$ is the geometric genus of the varieties parametrized by $\mathcal M_n$ and that, for each $n$ and $s=1$, one can choose a different $N$ satisfying \eqref{eq.lemmaemb}. Note also that \eqref{eq.lemmaemb} can be rephrase in terms of $r$, so the last statement follows. \textcolor{black}{Finally $\mathcal M_n$ is reduced since the simple cyclic cover constructed in Theorem ~\ref{theorem.embedding.a} is unobstructed by Lemma ~\ref{unobs}.
 } \QEDB\par

\smallskip

\textcolor{black}{Arguing as in the proof of Corollary~\ref{def of can morphisms},}  Theorem ~\ref{prop.birational} \textcolor{black}{yields} the following: 

\begin{corollary}\label{def of can morphisms birational}
\color{black}
Let  $m \geq 3$.  For each integer $n$, $n \geq 2$ such that 
\begin{equation*}
 2(n-1)(\floor{n/2}+1)-\floor{n/2} \leq m+2,   
\end{equation*}
 there is 
at least one reduced 
 and irreducible component $\mathcal M_n$ of the moduli space of varieties of general type of dimension $m$ such that:
\begin{enumerate}
    \item[(i)] $\mathcal M_n$ has a locally closed locus that parametrizes varieties whose canonical map $\varphi$ is a morphism, finite of degree $n$ onto a smooth image; more precisely, $\varphi$ factors through simple cyclic cover of degree $n$. 
    \item[(ii)] The general points of $\mathcal M_n$ 
    correspond to varieties whose canonical map is a finite birational morphism onto its image. 
\end{enumerate}  
For different values of $n$, the components $\mathcal M_n$ are different.
\end{corollary}

\textcolor{black}{\textcolor{black}{Theorem ~\ref{new.degreehalf} yields} the following:}

\color{black}
\begin{corollary}\label{def of can morphisms halved}
Let \textcolor{black}{$m \geq 6$}.
 For each even integer $n$ such that 
\begin{enumerate}
\item 
\begin{equation*}
 4 \leq n \leq (2m+8)/5; 
\end{equation*}
\item and there exists an integer $\nu$ satisfying
\begin{equation*}
\textcolor{black}{2\frac{m+4}{n}  \leq \nu \leq 2\frac{m-1}{n-2},}
\end{equation*}
\end{enumerate}
there is 
at least one reduced, irreducible, uniruled component $\mathcal M_n$ of the moduli space of varieties of general type of dimension $m$ such that:
\begin{enumerate}
    \item[(i)] $\mathcal M_n$ has a locally closed locus that parametrizes varieties whose canonical map $\varphi$ is a morphism, finite of degree $n$ onto  a smooth image; more precisely, $\varphi$ factors through a $\mathbb{Z}_{n/2}\times\mathbb{Z}_2$ Galois cover. 
    \item[(ii)] The general points of $\mathcal M_n$ 
    correspond to varieties whose canonical map is a flat morphism, finite and of degree $n/2$ onto a smooth image. 
\end{enumerate}  
For different values of $n$, the components $\mathcal M_n$ are different.
\end{corollary}

\begin{proof}
\color{black} (1) and (2) of the statement imply, when 
$s=1$, formula \eqref{eq.remark.3}. If $s=1$, \eqref{eq.remark.4}
becomes 
\begin{equation}\label{eq.cor11.3}
m+1 \leq N\leq 2m -5n/2+5,
\end{equation} 
and (1) implies the existence of integers $N$ satisfying
\eqref{eq.cor11.3}. 
\color{black}
Then, 
arguing as in the proof of Corollary~\ref{def of can morphisms},
Theorem~\ref{new.degreehalf} implies the result. 
\textcolor{black}{By Theorem ~\ref{2:1} the deformation space of $X$ is smooth and uniruled. Since $H^0(T_X) = 0$, $X$ has finite automorphism group and hence the corresponding moduli component is also uniruled.} 
\end{proof}

\color{black}
\begin{remark}
{\rm The inequalities 
\begin{equation*}
    4 \leq n \leq 2(\sqrt{m+8}-2)
\end{equation*}
imply (1) and (2) of the statement of Corollary~\ref{def of can morphisms halved}.}
\end{remark}

\textcolor{black}{Arguing as in the proof of Corollary~\ref{def of can morphisms}, Theorem~\ref{prop.embedding.dihedral} yields the following:}

\color{black}
\begin{corollary}\label{def of can morphisms dn}
Let $m \geq 18$.  For each even integer $n$ such that $6 \leq n \leq \sqrt{2m}$, there is 
at least one reduced 
and irreducible component $\mathcal M_n$ of the moduli space of varieties of general type of dimension $m$ such that:
\begin{enumerate}
    \item[(i)] $\mathcal M_n$ has a locally closed locus that parametrizes varieties whose canonical map $\varphi$ is a morphism, finite of degree $n$ onto a smooth image; more precisely, $\varphi$ factors through simple dihedral cover of degree $n$. 
    \item[(ii)] The general points of $\mathcal M_n$ 
    correspond to varieties whose canonical map is an embedding. 
    \item[(iii)] \textcolor{black}{The canonical models of the varieties corresponding to general points of $\mathscr{M}_n$ degenerate to a rope of multiplicity $n$, along a general one-parameter family that intersects the special locus \textcolor{black}{described in (i)}.}

\end{enumerate}  
For different values of $n$, the components $\mathcal M_n$ are different.
Letting $m$ vary, the general points of the $\mathcal M_n$ parametrize canonically embedded varieties of any codimension $r \geq 5$.
\end{corollary}

\color{black}
\begin{remark}
As hinted in the statements of Corollaries~\ref{def of can morphisms},  \ref{def of can morphisms birational}, \ref{def of can morphisms halved} and \ref{def of can morphisms dn}, once we fix $m$ and $n$ in any of these corollaries, there could be more than one reduced and irreducible component $\mathcal M_n$, even if we also fix $p_g$. This is because we can make different choices for the multidegree of  $Y$ and  different choices for the ramification of $\pi$. For example, regarding 
Corollary~\ref{def of can morphisms}, for $m=6$, we can exhibit two components
$\mathcal M_2$ and $\mathcal M'_2$, the former constructed from a double cover $\pi$ of a $(2,4)$--complete intersection $Y$, branched along a smooth member of $|\mathcal O_Y(4)|$, and the latter constructed from a double cover $\pi'$ of a $(2,6)$--complete intersection $Y'$, branched along a smooth member of $|\mathcal O_{Y'}(2)|$. In both cases, $p_g=9$; however, $K^6=16$ in the former case and $K^6=24$ in the latter case.
\end{remark}

\section{\textcolor{black}{Open questions and final remarks}} \label{6}
 
We end this article by asking \textcolor{black}{some} questions \textcolor{black}{regarding the} subvarieties we have constructed. 
 
  \begin{question}\label{question.ci}
 Recall that the subvarieties obtained in these article when applying Theorem~\ref{theorem.embedding}  to cyclic covers (and, more generally, to suitable iterated cyclic covers) are one-parameter smoothing of ropes. 
{Proposition ~\ref{prop.ci} implies 
that general members of some of these smoothings are complete intersections, for example, when the hypotheses of Theorem~\ref{theorem.embedding.a} or Theorem~\ref{prop.embedding.bicyclic} are satisfied.
However,
Proposition ~\ref{prop.ci}} \textit{does not} 
imply that general members of every one–parameter smoothing of 
\textcolor{black}{those} ropes are complete intersections. Then, in the above 
cases, is every one–parameter smoothing
  a complete intersection or are there some  one-parameter deformations yielding non--complete intersections subvarieties? This question is quite intriguing, since Theorem~\ref{theorem.embedding.a} and Theorem~\ref{prop.embedding.bicyclic} produce {$m$-dimensional, smooth} subvarieties
  \textcolor{black}{of $\mathbb P^N$ of any codimension and, in particular,  in the range $m \geq N/2 +1$, where constructing non--complete intersection is quite difficult, and, more remarkably, \textcolor{black}{fall} also in the range of Hartshorne's conjecture.} 
 \end{question}

\begin{question}\label{question.ci.diherdral}
Are some of the smooth subvarieties $\varphi'(X')$ constructed in Theorem~\ref{prop.embedding.dihedral} by deforming dihedral covers, non--complete intersection subvarieties? Recall that these subvarieties are $s$-subcanonical (see \eqref{pushdn} and Proposition~\ref{prop.subcan.subvar}). Although their degree and $s$ are such that there exist smooth complete intersection, $s$-subcanonical subvarieties with the same degree in the same projective space, this doesn't necessarily implies that the subvarieties $\varphi'(X')$ are complete intersections. Since \textcolor{black}{some of them} have small codimension, in the same range mentioned in Question~\ref{question.ci} \textcolor{black}{(see a few cases of this in Examples~\ref{example.dihedral.1} and \ref{example.dihedral.2})}, the question is relevant. \color{black}
  \textcolor{black}{We note that the} methods \textcolor{black}{that work in the case of iterated cyclic covers} to prove that \textcolor{black}{a} general deformation is a complete intersection, do not work for dihedral covers.
\end{question}

 \begin{question} \textcolor{black}{Is the morphism $\varphi'$ of Example~\ref{enr} an embedding for the varieties appearing in the white rows of the table?} More generally, using our methods, one can construct infinitely many smooth varieties with birational canonical morphisms  which \textcolor{black}{would be} non--complete intersections {if} they 
 \textcolor{black}{were} embeddings. The question is to determine if there are any special one-parameter family for these covers along which they deform to embeddings. \textcolor{black}{Again, the question is relevant, for the images of many of these embeddings would be small codimension subvarieties.}
 \end{question}

 \color{black}
 
 \begin{question}
 It is well known that one can construct $m$-dimensional, smooth subvarieties in $\mathbb P^N$, in the range  $m < \frac{N}{2} +1$ as degeneracy loci of vector bundle homomorphisms. 
 The method we use in this paper to construct smooth subvarieties is completely different. Therefore, are 
 the smooth, non--complete intersection  subvarieties constructed by our methods degeneracy loci of vector bundle homomorphisms? Particularly, are the smooth, non--complete intersection subvarieties constructed in in Theorems~\ref{real.main0.5} and \ref{main0.5} degeneracy loci of vector bundle homomorphisms? 
 \end{question}
 
Next remark shows the subtleties of the situation we handle, as it 
 introduces an example of an unobstructed morphism $\varphi$ to 
 which we can associate an embedded rope which is obstructed. In some of the cases, 
even if the deformation space of $\varphi$ is smooth, the
rope lies in at least two components of the Hilbert scheme: 
 
\begin{remark}\label{f3}
With the notation of Set-up~\ref{setup1}, 
we give an example of the following situation:
\begin{itemize}
    \item[(1)] a morphism $\varphi: X \to \mathbb{P}^N$ with $Y$ smooth 
   such that there exists a smooth curve $T$ 
   with a distinguished point $0$ and a flat family of morphisms $\Phi: \mathscr{X} \to \mathbb{P}^N_T$ over $T$ such that $\Phi_0 = \varphi$ and $\Phi_t$ is an embedding for all $t \neq 0$; 
   \item[(2)] $(\textrm{Im}\Phi)_0$ is a rope $\widetilde{Y}$ 
   in
   $\mathbb{P}^N$ supported on $Y$ with conormal bundle $\mathscr{E}$;
    \item[(3)] the morphism $\varphi$ is unobstructed   
    and, by (1), a general element of the algebraic formally 
    semiuniversal deformation of $\varphi$ is an embedding;
    \item[(4)] However, the rope $\widetilde{Y}$ does not 
    correspond to a smooth point of its Hilbert scheme.  
\end{itemize} 
\smallskip

\noindent For such an example, consider the rational normal 
scroll $Y = \mathbb{F}_e \hookrightarrow \mathbb{P}^N$ with $e = 3$ or $e = 4$. By \cite{R76}, there exists a 
double $K3$ cover $\pi: X \to Y$, branched along a reasonably singular curve $C \in |-2K_Y|$ (see \cite[Theorem 2.2]{R76}). Then $H^1(\mathcal{N}_{\pi}) = H^1(\pi_*\mathcal{N}_{\pi}) = H^1(-2K_Y|_C) = 
H^1(2K_C)=0$ by Proposition ~\ref{pushnpi}. Since $H^1(\mathcal{N}_{Y/\mathbb{P}^N}) = H^1(\mathcal{N}_{Y/\mathbb{P}^N} \otimes K_Y) = 0$, we have $H^1(\mathcal{N}_{\varphi}) = 0$ and $\varphi$ is unobstructed. But $\widetilde{Y}$ is a singular point of its Hilbert scheme by  
\cite[Theorem 4.1]{GP97}. Moreover, 
if $N \geq 10$ and $N$ is congruent with $1$ modulo $4$, then the Hilbert point 
of $\widetilde{Y}$ lies in two irreducible components 
of the Hilbert scheme (see \cite[Theorem 4.3]{GP97}).

\end{remark}

\bibliographystyle{plain}

\end{document}